\documentclass[a4paper,11pt,reqno]{amsart}
\usepackage[letterpaper, margin=1 in]{geometry}
\usepackage{caption} 

\usepackage{amsthm}
\usepackage{bm}

\newtheorem{theorem}{Theorem}[section]

\newtheorem{lemma}[theorem]{Lemma}

\newtheorem{corollary}[theorem]{Corollary}

\newtheorem{proposition}[theorem]{Proposition}

\usepackage{accents}

\theoremstyle{definition}
\newtheorem{definition}[theorem]{Definition}

\newtheorem{remark}[theorem]{Remark}

\usepackage{tikz-cd}

\usepackage{adjustbox}
\usepackage{float}
\usepackage{amssymb,amsfonts,amsthm,amsmath,amsopn,amstext,amscd,latexsym,color}
\usepackage[all]{xy}
\usepackage{mathtools}

\usepackage{enumerate}

\xyoption{web}

\usepackage[utf8]{inputenc}

\newcommand{\bbA}{{\mathbb A}}

\newcommand{\bbG}{{\mathbb G}}

\usepackage{diagbox}
\usepackage{longtable}
\DeclareMathOperator{\Pic}{Pic}

\usepackage{hyperref}

\usepackage{tikz-cd}

\usepackage{mathtools}

\usepackage{longtable}
\usepackage{verbatim}

\usepackage[OT2,T1]{fontenc}
\DeclareSymbolFont{cyrletters}{OT2}{wncyr}{m}{n}
\DeclareMathSymbol{\Sha}{\mathalpha}{cyrletters}{"58}

\newcommand{\Gm}{\mathbb G_m} 
\usepackage{graphics}

\newcommand{\Q}{\mathbb{Q}}
\newcommand{\A}{\mathbb{A}}
\newcommand{\Z}{\mathbb{Z}}

\usepackage{enumitem}

\DeclareMathOperator{\Gal}{Gal}

\DeclareMathOperator{\Cor}{Cor}

\author{\sc André Macedo}

\address{André Macedo\\
Department of Mathematics and Statistics\\ University of Reading\\
Whiteknights, PO Box 220\\
Reading RG6 6AX\\
UK}
   \email{c.a.v.macedo@pgr.reading.ac.uk}
\urladdr{https://sites.google.com/view/andre-macedo}

\title{On the obstruction to the Hasse principle for multinorm equations}
    
\begin{document}

\begin{abstract}
  We investigate local-global principles for multinorm equations over a global field. To this extent, we generalize work of Drakokhrust and Platonov to provide explicit and computable formulae for the obstructions to the Hasse principle and weak approximation for multinorm equations. We illustrate the scope of this technique by extending results of Bayer-Fluckiger--Lee--Parimala \cite{eva}, Demarche--Wei \cite{demarche} and Pollio \cite{pollio}.
\end{abstract}

\maketitle
\tableofcontents
\section{Introduction}

Let $K$ be a global field and let $L=(L_1,\dots,L_n)$ be an $n$-tuple ($n \geq 1$) of finite separable extensions of $K$. In this paper, we study the so-called \emph{multinorm principle} for $L$, which is said to hold if, for any $c \in K^*$, the affine $K$-variety
\begin{equation}\label{eq:Xc}
    X_c : \prod\limits_{i=1}^{n} N_{L_i/K}(\Xi_i)=c
\end{equation}
\noindent satisfies the Hasse principle. In other words, $L$ satisfies the multinorm principle if, for all $c \in K^*$, the existence of points on $X_c$ over every completion of $K$ implies the existence of a $K$-point. 

From a geometric viewpoint, $X_c$ defines a principal homogenous space under the \emph{multinorm one torus} $T$, defined by the exact sequence of $K$-algebraic groups
$$ 1 \to T \to \prod\limits_{i=1}^{n} R_{L_i/K} \Gm \xrightarrow{\prod_i N_{L_i/K}} \Gm \to 1,$$
\noindent where $R_{L_i/K} \Gm$ denotes the Weil restriction of $\Gm$ from $L_i$ to $K$. In this way, the Tate-Shafarevich group $\Sha(T)$ of $T$ is naturally identified with the \textit{obstruction to the multinorm principle} for $L$, defined as
$$\mathfrak{K}(L,K)=K^* \cap \prod\limits_{i=1}^{n} N_{L_i/K}(\A^*_{L_i}) / \prod\limits_{i=1}^{n} N_{L_i/K}(L_i^*),$$
\noindent where $\A^{*}_{L_i}$ denotes the id\`{e}le group of $L_i$ and the multinorm principle holds if and only if $\mathfrak{K}(L,K)=1$.

In the toric case, the Hasse principle for principal homogeneous spaces is strikingly connected with \emph{weak approximation}. This property is said to hold for a torus $T$ over $K$ if the \textit{defect of weak approximation}
$$A(T) =\prod\limits_v T(K_v)/\overline{T(K)}$$
\noindent is trivial (here $\overline{T(K)}$ denotes the closure of $T(K)$ in $\prod_v T(K_v)$ with respect to the product topology). In \cite[\S11.6]{Vosk}, Voskresenski\u{\i} showed the existence of an exact sequence
\begin{equation}\label{eq:Vosk}
    0  \to A(T)  \to \operatorname{H}^1(K,\Pic \overline{X})^{\vee} \to \Sha (T) \to 0,
\end{equation}
\noindent where $X$ denotes a smooth compactification of $T$, $\overline{X}$ the base change of $X$ to an algebraic closure of $K$ and $\phantom{ }^\vee$ stands for the Pontryagin dual of an abelian group. 

Returning to the multinorm principle, when $n=1$ one recovers the classical \emph{Hasse norm principle} (HNP), a topic that has been extensively studied in the literature (see e.g. \cite[\S 6.3]{Platonov} or \cite[\S 1]{MN19} for a survey of known results). If $L/K$ is Galois, then there is an explicit description of the obstruction to the HNP (due to Tate in \cite[p. 198]{C-F}) in terms of the group cohomology of its local and global Galois groups. Drakokhrust later obtained (in \cite{Drak}) a more general description of this obstruction for an arbitrary extension $L/K$ in terms of generalized representation groups.

For $n > 1$, such a description has not yet been obtained. Nonetheless, multiple cases have been analyzed in the literature. For example, if $n=2$ it is known that the multinorm principle holds if 

\begin{enumerate}
    \item\label{hur} $L_1$ or $L_2$ is a cyclic extension of $K$ (\cite[Proposition 3.3]{Hur});
    \item\label{pras} $L_1/K$ is abelian, satisfies the HNP and $L_2$ is linearly disjoint from $L_1$ (\cite[Proposition 4.2]{Prar});
    \item\label{PR_res} the Galois closures of $L_1/K$ and $L_2/K$ are linearly disjoint over $K$ (\cite{PR}).
\end{enumerate}

\noindent Subsequent work of Demarche and Wei provided a generalization of the result in \eqref{PR_res} to $n$ extensions (\cite[Theorems 1 and 6]{demarche}), while also addressing weak approximation for the associated multinorm one torus. In \cite{pollio}, Pollio computed the obstruction to the multinorm principle for a pair of abelian extensions and, in \cite{eva}, Bayer-Fluckiger, Lee and Parimala provided sufficient and necessary conditions for the multinorm principle to hold assuming that one of the extensions $L_i/K$ is cyclic.

In this paper, we provide an explicit description of the obstructions to the multinorm principle and weak approximation for the multinorm one torus of $n$ arbitrary extensions. To achieve this, we generalize the concept (due to Drakokhrust and Platonov in \cite{DP}) of the \emph{first obstruction to the Hasse principle} (see Section \ref{sec:1st_obs}). By then adapting work of Drakokhrust (\cite{Drak}), we obtain our main result (Theorem \ref{thm:main_result}), describing the obstructions to the multinorm principle and weak approximation for the multinorm one torus in terms of generalized representation groups of the relevant local and global Galois groups. The formulas given in Theorem \ref{thm:main_result} are effectively computable and we provide algorithms in GAP \cite{gap} for this effect (see Remark \ref{rem:finite_time2}).

We also apply our techniques to describe the validity of the local-global principles in three concrete examples (see Section \ref{sec:applications}). We start by proving a result inspired by \cite[Theorem 6]{demarche} that compares the birational invariants $\operatorname{H}^1(K,\Pic \overline{X})$ and $\operatorname{H}^1(K,\Pic \overline{Y})$, where $Y$ is a smooth compactification of the norm one torus $S=R^1_{F/K} \bbG_m$ of the extension $F=\bigcap\limits_{i=1}^{n} L_i$. In particular, we show (Theorem \ref{thm:demarch_wei_thm}) that under certain conditions there is an isomorphism $$\operatorname{H}^1(K,\Pic \overline{X}) \xrightarrow{\simeq}\operatorname{H}^1(K,\Pic \overline{Y}).$$

\noindent This results further allows us to compare the defect of weak approximation for $T$ with the defect of weak approximation for $S$ (Corollary \ref{cor:dem1}). 

Under the same assumptions, we also show (Theorem \ref{thm:pollio}) the existence of isomorphisms $$\mathfrak{K}(L,K) \cong \mathfrak{K}(F/K) \textrm{ and } A(T) \cong A(S)$$
\noindent when all the extensions $L_i/K$ are abelian. This theorem generalizes Pollio's result (in \cite{pollio}) on the obstruction to the multinorm principle for a pair of abelian extensions. 

In Section \ref{sec:eva} we complement \cite[Theorem 8.3]{eva} by providing a characterization (Theorem \ref{thm:eva}) of weak approximation for the multinorm one torus of $n$ non-isomorphic cyclic extensions of prime degree $p$. More precisely, we show that both the multinorm principle and weak approximation for $T$ hold if $[L_1 \dots L_n : K]>p^2$. Otherwise, weak approximation holds if and only if the multinorm principle fails (a property that can be detected by precise local conditions, see Remark \ref{rem:eva}). While preparing this paper, we became aware of the recent (and independent) work of Lee \cite{lee}, who extends results of \cite[\S8]{eva} to provide a description of the multinorm principle and weak approximation for the multinorm one torus of $n$ non-isomorphic cyclic extensions (and, in this way, obtains a result more general than Theorem \ref{thm:eva}).

\subsection*{Notation}

Given a global field $K$, we denote its set of places by $\Omega_K$. For $v \in \Omega_K$, we use the notation $K_v$ for the completion of $K$ at $v$ and, if $L$ is a Galois extension of $K$, we denote by $G_v$ a choice of decomposition group of $L/K$ at $v $. 

Given a finite group $G$, a subgroup $H$ of $G$, a $G$-module $A$, an integer $q$ and a prime number $p$, we use the notation:

\begin{longtable}{p{1.5cm} p{14cm}}
$|G|$ & the order of $G$\\
$Z(G)$ & the center of $G$\\
$[H,G]$ & the subgroup of $G$ generated by all commutators $[h,g]$ with $h \in H,g \in G$\\
$\Phi^{G}(H)$ & the subgroup of $H$ generated by all commutators $[h,g]$ with $h \in H \cap g H g^{-1},g \in G$\\
$G^{ab}$ & the abelianization $G/[G,G]$ of $G$\\
$G_p$ & a Sylow $p$-subgroup of $G$\\
$\hat{\operatorname{H}}^q(G,A)$ & the $q$-th Tate cohomology group
\end{longtable}

\noindent We also often use the notation $G'$ for the derived subgroup $[G,G]$ of $G$. If $H$ is a normal subgroup of $G$, we write $H \trianglelefteq G$. For $x,y \in G$ we adopt the convention $[x,y]=x^{-1}y^{-1}xy$ and $x^y=y^{-1}xy$. If $G$ is abelian, we denote its $p$-primary part by $G_{(p)}$.

\subsection*{Acknowledgements}
I would like to thank Prof. Eva Bayer-Fluckiger for a conversation that motivated this work and my supervisor Rachel Newton for useful discussions on the manuscript and for pointing out the recent preprint \cite{lee}. This work was supported by the FCT doctoral scholarship SFRH/BD/117955/2016.

\section{The first obstruction to the multinorm principle}\label{sec:1st_obs}

In this section we define the concept of the first obstruction to the multinorm principle and present several of its properties. We fix a global field $K$, an $n$-tuple $L=(L_1,\dots,L_n)$ of finite separable extensions of $K$ and a Galois extension $N/K$ containing all the fields $L_1,\dots,L_n$. We denote $G=\Gal(N/K)$, $H_i=\Gal(N/L_i)$ for $i=1,\dots,n$ and $H=\langle H_1,\dots,H_n \rangle$, the subgroup of $G$ generated by all the $H_i$. Note that $H=\Gal(N/F)$, where $F=\bigcap\limits_{i=1}^{n} L_i$.

\begin{definition}
We define the \emph{first obstruction to the multinorm principle for $L$ corresponding to $(N,L,K)$} as
$$\mathfrak{F}(N,L,K)=K^* \cap \prod\limits_{i=1}^{n} N_{L_i/K}(\A^{*}_{L_i}) / \prod\limits_{i=1}^{n} N_{L_i/K}(L_i^*)(K^* \cap N_{N/K}(\A^{*}_{N})).$$
\end{definition}

\begin{remark}
This notion generalizes the concept (introduced by Drakokhrust and Platonov in \cite{DP}) of the \emph{first obstruction to the Hasse principle for $L/K$ corresponding to a tower of fields $N/L/K$}, defined as $\mathfrak{F}(N/L/K)=K^* \cap  N_{L/K}(\A^{*}_{L}) /  N_{L/K}(L^*)(K^* \cap N_{N/K}(\A^{*}_{N})).$
\end{remark}

The first obstruction to the multinorm principle has various useful properties -- for example, it is clear from the definition that the total obstruction to the multinorm principle $\mathfrak{K}(L,K)$ surjects onto $\mathfrak{F}(N,L,K)$ with equality if the Hasse norm principle holds for $N/K$. Moreover, this equality also happens if the first obstruction to the Hasse principle for some extension $L_i/K$ coincides with the total obstruction to the Hasse norm principle $\mathfrak{K}(L_{i}/K)=K^* \cap  N_{L_i/K}(\A^{*}_{L_i}) / N_{L_i/K}(L_i^*)$ (called the \emph{knot group} of $L_i/K$):

\begin{lemma}\label{lem:1stobs_equal_multiknot}
If $\mathfrak{K}(L_{i}/K)=\mathfrak{F}(N/L_{i}/K) $ for some $i =1,\dots,n$, then $\mathfrak{K}(L,K)=\mathfrak{F}(N,L,K)$.
\end{lemma}

\begin{proof}
The assumption translates into $K^* \cap N_{N/K}(\A^*_N) \subset N_{L_{i}/K}(L_{i}^*) $. This implies that \newline ${\prod\limits_{i=1}^{n}N_{L_i/K}(L_i^*)(K^* \cap N_{N/K}(\A^*_N)) = \prod\limits_{i=1}^{n} N_{L_i/K}(L_i^*)}$ and hence $\mathfrak{K}(L,K)=\mathfrak{F}(N,L,K)$.\end{proof}

\begin{corollary}\label{cor:square_free}
If $[L_{i}:K]$ is square-free for some $i = 1,\dots,n$, then $\mathfrak{K}(L,K)=\mathfrak{F}(N,L,K)$. 
\end{corollary}

\begin{proof}
By \cite[Corollary 1]{DP}, if $[L_{i}:K]$ is square-free, then $ \mathfrak{K}(L_{i}/K)=\mathfrak{F}(N/L_{i}/K)$. Now apply Lemma \ref{lem:1stobs_equal_multiknot}.
\end{proof}

More generally, one has the following criterion (extending \cite[Theorem 3]{DP}) for the equality $\mathfrak{K}(L,K)=\mathfrak{F}(N,L,K)$.

\begin{proposition}
\label{prop:1stobs_cor}
Let $k_1,\dots,k_n$ be positive integers. For each $i=1,\dots,n$, choose a collection of $k_i$ subgroups $G_{i,1},\dots,G_{i,k_i}$ of $G$ and $k_i$ subgroups $H_{i,1},\dots,H_{i,k_i}$ such that $H_{i,j} \subset H_i \cap G_{i,j}$ for any $j=1,\dots, k_i$. Set $L_{i,j}=N^{H_{i,j}}$ and $K_{i,j}=N^{G_{i,j}}$ for all $i,j$. Suppose that the Hasse norm principle holds for all the extensions $L_{i,j}/K_{i,j}$ and that the map 
$$\bigoplus\limits_{i=1}^{n}\bigoplus\limits_{j=1}^{k_i} \Cor^{G}_{G_{i,j}}:\bigoplus\limits_{i=1}^{n}\bigoplus\limits_{j=1}^{k_i}  \hat{\operatorname{H}}^{-3}(G_{i,j},\Z)  \to \hat{\operatorname{H}}^{-3}(G,\Z)   $$

\noindent is surjective. Then $ \mathfrak{K}(L,K)=\mathfrak{F}(N,L,K)$. 

\end{proposition}

\begin{proof}
The statement follows from an argument analogous to the one given by Drakokhrust and Platonov for the Hasse norm principle case, see \cite[Theorem 3]{DP}.
\end{proof}

A further trait of the first obstruction to the multinorm principle $\mathfrak{F}(N,L,K)$ is that it can be expressed in terms of the local and global Galois groups of the towers $N/L_i/K$ (in similar fashion to the first obstruction to the Hasse norm principle). In order to prove this, we mimic the work Drakokhrust and Platonov in \cite[\S 2]{DP}. We will use the following lemma:

\begin{lemma}\cite[Lemma 1]{DP}\label{lem1DP} Let $N/L/K$ be a tower of global fields with $N/K$ Galois. Set $G=\Gal(N/K)$ and $H=\Gal(N/L)$. Then, given a place $v$ of $K$, the set of places $w$ of $L$ above $v$ is in bijection with the set of double cosets in the decomposition $G = \bigcup\limits_{i=1}^{r_v} H x_i G_v$. If $w$ corresponds to $H x_{i} G_v$, then the decomposition group $H_w$ of the extension $N/L$ at ${w}$ equals $H \cap x_{i} G_v x_{i}^{-1}$.
\end{lemma}

In our situation, for any $v \in \Omega_K$ and $i=1,\dots,n$, let $G=\bigcup\limits_{k=1}^{r_{v,i}} H_i x_{i,k} G_v$ be a double coset decomposition. By the above lemma, $H_{i,w}:= H_i \cap x_{i,k} G_v x_{i,k}^{-1}$ is the decomposition group of $N/L_i$ at a place $w$ of $L_i$ above $v$ corresponding to the double coset $H_i x_{i,k} G_v$. Now consider the commutative diagram:
\begin{equation}\label{diag:1stobs_defn}
\xymatrix{
\bigoplus\limits_{i=1}^{n} {H}_i^{\textrm{ab}}  \ar[r]^{{\psi}_1} & {G}^{\textrm{ab}}\\
\bigoplus\limits_{i=1}^{n}(\bigoplus\limits_{v \in \Omega_K} ( \bigoplus\limits_{w|v} {{H}_{i,w}^{\textrm{ab}} }) ) \ar[r]^{\ \ \ \ \ \ \ {\psi}_2} \ar[u]^{{\varphi}_1 }&\bigoplus\limits_{v \in \Omega_K}{{G}_v^{\textrm{ab}} }\ar[u]_{\varphi_2}
}
\end{equation}

\noindent Here the superscript$\phantom{a}^{\textrm{ab}}$ above a group denotes its abelianization and the inside sum over $w|v$ runs over all the places $w$ of $L_i$ above $v$. Additionally, the maps $\varphi_1,\psi_1$ and $\varphi_2$ are induced by the inclusions $H_{i,w} \hookrightarrow H_i, H_i \hookrightarrow G$ and $G_v \hookrightarrow G$, respectively, while $\psi_2$ is obtained from the product of all conjugation maps $H_{i,w}^{ab} \to G_v^{ab}$ sending $h_{i,k} [H_{i,w},H_{i,w}]$ to $x_{i,k}^{-1} h_{i,k} x_{i,k} [G_v,G_v]$. We denote 
by ${\psi}_2^{v}$ (respectively, ${\psi}_2^{nr}$) the restriction of the map ${\psi}_2$ to the subgroup $\bigoplus\limits_{i=1}^{n} ( \bigoplus\limits_{w|v} {{H}_{i,w}^{\textrm{ab}} }) $ (respectively, $\bigoplus\limits_{i=1}^{n}(\bigoplus\limits_{\substack{v \in \Omega_K \\ v \text{ unramified}}} ( \bigoplus\limits_{w|v} {{H}_{i,w}^{\textrm{ab}} }) )$). With this notation set, we can now establish the main result of this section (generalizing \cite[Theorem 1]{DP}):

\begin{theorem}\label{thm:thm1DP_gen}
In the notation of diagram \eqref{diag:1stobs_defn}, we have
$$\mathfrak{F}(N,L,K) \cong \ker\psi_1/\varphi_1(\ker\psi_2).$$
\end{theorem}

\begin{proof}
Diagram \eqref{diag:1stobs_defn} can be written as
\begin{equation}\label{diag:1stobs_defn2}
\xymatrix{
\bigoplus\limits_{i=1}^{n} \hat{\operatorname{H}}^{-2}({H}_i,\Z)  \ar[r]^{{\psi}_1} & \hat{\operatorname{H}}^{-2}(G,\Z)\\
\bigoplus\limits_{i=1}^{n}(\bigoplus\limits_{v \in \Omega_K} ( \bigoplus\limits_{w|v} { \hat{\operatorname{H}}^{-2}(H_{i,w},\Z) }) ) \ar[r]^{\ \ \ \ \ \ \ {\psi}_2} \ar[u]^{{\varphi}_1}&\bigoplus\limits_{v \in \Omega_K}{ \hat{\operatorname{H}}^{-2}(G_v,\Z) }\ar[u]_{\varphi_2}
}
\end{equation}

\noindent By the local (respectively, global) Artin isomorphism, we have $\hat{\operatorname{H}}^{-2}(H_{i,w},\Z) \cong \hat{\operatorname{H}}^{0}(H_{i,w},N_w^*)$ and $\hat{\operatorname{H}}^{-2}(G_v,\Z) \cong \hat{\operatorname{H}}^{0}(G_v,N_v^*)$ (respectively, $\hat{\operatorname{H}}^{-2}({H}_i,\Z) \cong \hat{\operatorname{H}}^{0}(H_i,C_N)$ and $\hat{\operatorname{H}}^{-2}(G,\Z) \cong \hat{\operatorname{H}}^{0}(G,C_N)$, where $C_N$ is the idèle class group of $N/K$). Additionally, by \cite[Proposition 7.3(b)]{C-F} there are identifications $\bigoplus\limits_{v \in \Omega_K}( \bigoplus\limits_{w|v} \hat{\operatorname{H}}^{0}(H_{i,w},N_w^*)) \cong  \hat{\operatorname{H}}^{0}(H_{i},\A^{*}_N)$ and $\bigoplus\limits_{v \in \Omega_K} \hat{\operatorname{H}}^{0}(G_v,N_v^*) \cong  \hat{\operatorname{H}}^{0}(G,\A_{N}^{*})$. In this way, an argument analogous to the one given in \cite[\S 2]{DP} for the $n=1$ case shows that diagram \eqref{diag:1stobs_defn2} induces the commutative diagram
\begin{equation}\label{diag:1stobs_defn3}
\xymatrix{
\bigoplus\limits_{i=1}^{n} \hat{\operatorname{H}}^{0}({H}_i,C_N)  \ar[r]^{\ {\psi}_1} & \hat{\operatorname{H}}^{0}(G,C_N)\\
\bigoplus\limits_{i=1}^{n} \hat{\operatorname{H}}^{0}(H_i,\bbA^*_{N})  \ar[r]^{\  {\psi}_2} \ar[u]^{{\varphi}_1}& \hat{\operatorname{H}}^{0}(G,\bbA^*_{N}) \ar[u]_{\varphi_2}
}
\end{equation}

\noindent where $\varphi_{1},\varphi_2$ are the natural projections and $\psi_{1},\psi_2$ are induced by the product of the norm maps $N_{L_i/K}$. Using the definition of the cohomology group $\hat{\operatorname{H}}^0$, this diagram is equal to
\begin{equation}\label{diag:1stobs_defn4}
\xymatrix{
\bigoplus\limits_{i=1}^{n} \frac{\A_{L_i}^{*}}{L_i^{*} N_{N/L_i}(\A^{*}_{N})}  \ar[r]^{{\psi}_1} & \frac{\A_{K}^{*}}{K^{*} N_{N/K}(\A^{*}_{N})}\\
 \bigoplus\limits_{i=1}^{n} \frac{\A_{L_i}^{*}}{ N_{N/L_i}(\A^{*}_{N})} \ar[r]^{\ \ \  {\psi}_2} \ar[u]^{{\varphi}_1}&{ \frac{\A_{K}^{*}}{ N_{N/K}(\A^{*}_{N})} }\ar[u]_{\varphi_2}
}
\end{equation}

\noindent From diagram \eqref{diag:1stobs_defn4}, it is clear that
$$\ker \psi_1 = \{ (x_i L_i^*  N_{N/L_i}(\A^{*}_{N}))_{i=1}^{n} | \prod\limits_{i=1}^{n} N_{L_i/K}(x_i) \in K^* N_{N/K}(\A^{*}_N)\}$$ 
and
$$\varphi_1(\ker \psi_2) = \{ (x_i L_i^* N_{N/L_i}(\A^{*}_{N}))_{i=1}^{n} | \prod\limits_{i=1}^{n} N_{L_i/K}(x_i) \in N_{N/K}(\A^{*}_N)\}.$$ 

\noindent Now define
\begin{align*}
  f \colon \ker \psi_1 / \varphi_1(\ker \psi_2) &\longrightarrow \mathfrak{F}(N,L,K) \\
  (x_i L_i^* N_{N/L_i}(\A^{*}_{N}))_{i=1}^{n} &\longmapsto x \prod\limits_{i=1}^{n} N_{L_i/K}({L_i}^{*}) (K^* \cap N_{N/K}(\A^{*}_{N}) )
\end{align*}

\noindent where $x$ is any element of $ K^* \cap \prod\limits_{i=1}^{n} N_{L_i/K}(\A^{*}_{L_i})$ such that $\prod\limits_{i=1}^{n} N_{L_i/K}(x_i) \in x N_{N/K}(\A^{*}_N)$. It is straightforward to check that $f$ is well defined and an isomorphism.\end{proof}

\begin{remark}\label{rem:finite_time}
Given the knowledge of the local and global Galois groups of the towers $N/L_i/K$, the first obstruction to the multinorm principle can be computed in finite time by employing Theorem \ref{thm:thm1DP_gen}. First, it is clear that the computation of the groups $\ker \psi_1$ and $\varphi_1(\ker \psi_2^{v})$ for the ramified places $v$ of $N/K$ is finite. Moreover, from the definition of the maps in diagram \eqref{diag:1stobs_defn}, it is clear that if $v_1,v_2 \in \Omega_K$ are such that $G_{v_1}=G_{v_2}$, then $\varphi_1(\ker \psi_2^{v_1}) = \varphi_1(\ker \psi_2^{v_2})$. This shows that the computation of $\varphi_1(\ker \psi_2^{nr})$ is also finite. On this account, we designed a function in GAP \cite{gap} (whose code is available in \cite{macedo_code}) that takes as input the Galois groups $G, H_i$ and the decomposition groups $G_v$ at the ramified places of $N/K$ and outputs the group $\mathfrak{F}(N,L,K)$.\end{remark}

We conclude this section by providing two results that further reduce the amount of calculations necessary to compute $\mathfrak{F}(N,L,K)$ via Theorem \ref{thm:thm1DP_gen}. These are inspired by the same properties of the first obstruction to the Hasse norm principle (in \cite[\S 3]{DP}) and proved in the same way.

\begin{lemma}{\cite[Lemma 2]{DP}}
Let $v_1,v_2 \in \Omega_K$ be such that $G_{v_2} \subset G_{v_1}$. Then, in the notation of diagram \eqref{diag:1stobs_defn}, we have $$\varphi_1(\ker \psi_2^{v_2}) \subset \varphi_1(\ker \psi_2^{v_1}).$$
\end{lemma}

\begin{lemma}{\cite[Lemma 3]{DP}}
Let $v_1,v_2 \in \Omega_K$ be such that $G_{v_1} M = G_{v_2} M$ for some subgroup $M \subset Z(G) \cap \bigcap\limits_{i=1}^{n} H_i$. Then, in the notation of diagram \eqref{diag:1stobs_defn}, we have $$\varphi_1(\ker \psi_2^{v_1})= \varphi_1(\ker \psi_2^{v_2}).$$
\end{lemma}

\section{Generalized representation groups}\label{sec:gen_gps}

In this section we prove that the obstruction to the multinorm principle for $L$ can always be expressed in terms of the arithmetic of the extensions $L_i/K$ by using generalized representation groups (see Definition \ref{gen_rep_gp_defn} below) of $G=\Gal(N/K)$. Once again, many of the results in this section are inspired by and generalize Drakokhrust's work \cite{Drak} on the Hasse norm principle.

\begin{definition}\label{gen_rep_gp_defn}
Let $G$ be a finite group. A finite group $\overline{G}$ is called a \emph{generalized representation group} of $G$ if there exists a central extension 
$$1 \to M \to \overline{G} \xrightarrow[]{\lambda} G \to 1$$
\noindent such that $M \cap [\overline{G},\overline{G}] \cong \hat{\operatorname{H}}^{-3}(G,\Z)$. We call $M$ the base normal subgroup of $\overline{G}$. If in addition $M \subset [\overline{G},\overline{G}]$, we say that $\overline{G}$ is a \textit{Schur covering group of $G$}.

\end{definition}

\begin{proposition}\label{prop:1st=knot}
There exists a Galois extension $P/K$ containing $N$ and such that $$ \mathfrak{F}(P,L,K)=\mathfrak{K}(L,K).$$
\noindent Furthermore, this extension has the property that $\overline{G}=\Gal(P/K)$ is a generalized representation group of $G$ with base normal subgroup $\overline{M}=\Gal(P/N)$ and if $\overline{\lambda}:\overline{G} \to G$ is the associated projection map, we have $\Gal(P/L_i)=\overline{\lambda}^{-1}(H_i)$.
\end{proposition}

\begin{proof}
It follows from the proof of \cite[Lemma 1]{Drak} (see also \cite[Satz 3]{opolka}) that there exists a Galois extension $P/K$ such that the first obstruction to the Hasse norm principle $ \mathfrak{F}(P/L_i/K)$ coincides with the knot group $\mathfrak{K}(L_i/K)$ for all $L_i \in L$. Now apply Lemma \ref{lem:1stobs_equal_multiknot}. The stated properties of $P/K$ are shown in the references given above.
\end{proof}

As remarked in \cite{Drak}, the extension $P/K$ is not uniquely determined and the computation of its arithmetic is not always easy. Nonetheless, one can still compute $\mathfrak{F}(P,L,K)$ by commencing with an arbitrary generalized represention group of $G$.

Let $\widetilde{G}$ be any generalized representation group of $G$ with projection map $\widetilde{\lambda}$ and base normal subgroup $\widetilde{M}$. For any subgroup $B$ of $G$, define $\widetilde{B}=\widetilde{\lambda}^{-1}(B)$. We will use the following auxiliary lemma:

\begin{lemma}\label{lem:tau_1}
There exists an isomorphism $$\tau:[\widetilde{G},\widetilde{G}] \xrightarrow[]{\simeq} [\overline{G},\overline{G}]$$

\noindent with the following properties:

\begin{enumerate}[label=(\roman{*})]
    \item\label{prop_tau1} $\overline{\lambda}(\tau(a))=\widetilde{\lambda}(a)$ for every $a \in [\widetilde{G},\widetilde{G}]$;
    \item\label{prop_tau2} $\tau([\widetilde{g}_1,\widetilde{g}_2])=[\overline{g}_1,\overline{g}_2]$ for all $\widetilde{g}_1,\widetilde{g}_2 \in \widetilde{G}$ and $\overline{g}_1,\overline{g}_2 \in \overline{G}$ such that $\widetilde{\lambda}(\widetilde{g}_i)=\overline{\lambda}(\overline{g}_i)$.
\end{enumerate}

\noindent For any subgroup $B$ of $G$, $\tau$ further identifies

\begin{itemize}
\item $[\widetilde{B},\widetilde{B}] \cong [\overline{B},\overline{B}]$ and
    \item $\widetilde{M} \cap [\widetilde{B},\widetilde{B}] \cong \overline{M} \cap [\overline{B},\overline{B}].$ 
\end{itemize}

\end{lemma}

\begin{proof}
The isomorphism $\tau$ is constructed in \cite[Theorems 2.4.6(iv) and 2.5.1(i)]{kar} and the stated properties are clear from this construction. The additional identifications follow from \ref{prop_tau1} and \ref{prop_tau2}.\end{proof}

Let $R$ be the set of ramified places of $N/K$. For any $v \in \Omega_K$, set
\[\widetilde{S}_v=\begin{cases}
\widetilde{G_v}\textrm{, if $v \in R$,}\\
\textrm{a cyclic subgroup of } \widetilde{G_v} \textrm{ such that } \widetilde{\lambda}(\widetilde{S}_v) = G_v \textrm{, otherwise.}
\end{cases}\]

\noindent Furthermore, by the Chebotarev density theorem we can (and do) choose the subgroups $\widetilde{S}_v$ for $v \not\in R$ in such a way that all the cyclic subgroups of $\widetilde{G_v}$ such that $\widetilde{\lambda}(\widetilde{S}_v) = G_v$ occur. 
\begin{remark}
As pointed out in \cite[p. 31]{Drak}, a double coset decomposition $\overline{G}=\bigcup\limits_{k=1}^{r_{v,i}} \overline{H_i} \overline{x}_{i,k}  \overline{G_v}$ corresponds to a double coset decomposition $\widetilde{G}=\bigcup\limits_{k=1}^{r_{v,i}} \widetilde{H_i} \widetilde{x}_{i,k}  \widetilde{S}_v$, where $\widetilde{x}_{i,k}$ are any elements of $\widetilde{G}$ such that $\widetilde{\lambda}(\widetilde{x}_{i,k} )=\overline{\lambda}(\overline{x}_{i,k} )$.
\end{remark}

Consider the following diagram analogous to \eqref{diag:1stobs_defn}:
\begin{equation}\label{diag:1stobs_defn_generalized}
\xymatrix{
\bigoplus\limits_{i=1}^{n} \widetilde{H_i}^{\textrm{ab}}  \ar[r]^{\widetilde{\psi}_1} & \widetilde{G}^{\textrm{ab}}\\
\bigoplus\limits_{i=1}^{n}(\bigoplus\limits_{v \in \Omega_K} ( \bigoplus\limits_{w|v} {\widetilde{H}_{i,w}^{\textrm{ab}} }) ) \ar[r]^{\ \ \ \ \ \ \ \widetilde{\psi}_2} \ar[u]^{\widetilde{\varphi}_1}&\bigoplus\limits_{v \in \Omega_K}{\widetilde{S}_v^{\textrm{ab}} }\ar[u]_{\widetilde{\varphi}_2}
}
\end{equation}

\noindent where $\widetilde{H}_{i,w}=\widetilde{H_i} \cap \widetilde{x}_{i,k} \widetilde{S}_v \widetilde{x}_{i,k}^{-1}$ and all the maps are defined as in diagram \eqref{diag:1stobs_defn}. 

We now prove the main result of this section, namely that the object $\ker \widetilde{\psi}_1 / \widetilde{\varphi}_1(\ker \widetilde{\psi}_2)$ does not depend on the choice of generalized representation group (and thus, by Theorem \ref{thm:thm1DP_gen} and Proposition \ref{prop:1st=knot}, it always coincides with $\mathfrak{K}(L,K)$). Before we show this, we need a lemma. To ease the notation, we often omit the cosets $\widetilde{H_i}'$ and $\overline{H_i}'$ when working with elements of $\ker \widetilde{\psi}_1$ or $\ker \overline{\psi}_1$.

\begin{lemma}\label{lem:simpl_inters}
For any indices $1 \leq i_1 < i_2 \leq n$ and any $m \in \widetilde{H_{i_1}} \cap\widetilde{H_{i_2}} $, we have
$$h=(1,\dots,\underbrace{m}_{i_1\textrm{-th entry}}, 1,\dots, 1, \underbrace{m^{-1}}_{i_2\textrm{-th entry}}, 1,\dots, 1) \in \widetilde{\varphi}_1(\ker \widetilde{\psi}_2^{nr}).$$
\end{lemma}

\begin{proof}
We construct a vector $\alpha \in  \bigoplus\limits_{i=1}^{n}(\bigoplus\limits_{\substack{v \in \Omega_K \\ v \text{ unramified}}}( \bigoplus\limits_{w|v} {\widetilde{H}_{i,w}^{\textrm{ab}} }) )$ such that $ \widetilde{\psi}_2(\alpha)=1$ and $\widetilde{\varphi}_1(\alpha)=h$. Let $v$ be an unramified place of $K$ such that $\widetilde{S}_v = \langle m \rangle$. By definition, if $\widetilde{G}=\bigcup\limits_{k=1}^{r_{v,i}} \widetilde{H_i} \widetilde{x}_{i,k} \widetilde{S}_v$ is a double coset decomposition of $\widetilde{G}$, then $\widetilde{H}_{i,w}= \widetilde{H_i} \cap \widetilde{x}_{i,k} \widetilde{S}_v \widetilde{x}_{i,k}^{-1}$. Let us suppose, without loss of generality, that $\widetilde{x}_{i_1,{k_1}}=1=\widetilde{x}_{i_2,{k_2}}$ for some index $1 \leq k_1 \leq r_{v,i_1}$ (respectively, $1 \leq k_2 \leq r_{v,i_2}$) corresponding to a place $w_1 \in \Omega_{L_{i_1}}$ (respectively, $w_2 \in \Omega_{L_{i_2}}$) via Lemma \ref{lem1DP}. In this way, we have $m \in \widetilde{H}_{i_1,w_1}$ and $m^{-1} \in \widetilde{H}_{i_2,w_2}$. Setting the $(i_1,v,w_1)$-th (respectively, $(i_2,v,w_2)$-th) entry of $\alpha$ to be equal to $m$ (respectively, $m^{-1}$) and all other entries equal to $1$, we obtain $
\widetilde{\psi}_2(\alpha)=1$ and $\widetilde{\varphi}_1(\alpha)=h$.
\end{proof}

\begin{theorem}\label{thm:main_knot}

In the notation of diagram \eqref{diag:1stobs_defn_generalized}, we have
$$\mathfrak{K}(L,K) \cong \ker \widetilde{\psi}_1 / \widetilde{\varphi}_1(\ker \widetilde{\psi}_2).$$
\end{theorem}

\begin{proof}
By Theorem \ref{thm:thm1DP_gen} and Proposition \ref{prop:1st=knot}, we have $\mathfrak{K}(L,K) \cong \ker \overline{\psi}_1 / \overline{\varphi}_1(\ker \overline{\psi}_2)$, where the $\overline{\phantom{a}}$ notation is as in diagram \eqref{diag:1stobs_defn_generalized} with respect to the groups of Proposition \ref{prop:1st=knot}. Therefore, it suffices to prove that $$\ker \widetilde{\psi}_1 / \widetilde{\varphi}_1(\ker \widetilde{\psi}_2) \cong \ker \overline{\psi}_1 /  \overline{\varphi}_1(\ker \overline{\psi}_2).$$

\noindent Define
\begin{align*}
  f \colon \ker \widetilde{\psi}_1 / \widetilde{\varphi}_1(\ker \widetilde{\psi}_2) &\longrightarrow \ker \overline{\psi}_1 /  \overline{\varphi}_1(\ker \overline{\psi}_2) \\
  (\widetilde{h}_1 ,\dots,\widetilde{h}_n  ) &\longmapsto  (\overline{h}_1 ,\dots, \overline{h}_n )
\end{align*}

\noindent where, for each $i=1,\dots,n$, the element $\overline{h}_i \in \overline{H_i}$ is selected as follows: 
take $\overline{h}_i \in \overline{H_i}$ such that $\overline{\lambda}(\overline{h}_i)=\widetilde{\lambda}(\widetilde{h}_i)$ (note that $\overline{h}_i$ is only defined modulo ${\overline{M}}=\ker \overline{\lambda}$). In this way, we have $\overline{\lambda}(\overline{h}_1 \dots \overline{h}_n)=\widetilde{\lambda}(\widetilde{h}_1 \dots \widetilde{h}_n)$. Additionally, by Lemma \ref{lem:tau_1}\ref{prop_tau1}, $\overline{\lambda}(\tau(\widetilde{h}_1 \dots \widetilde{h}_n))=\widetilde{\lambda}(\widetilde{h}_1 \dots \widetilde{h}_n)$ and thus 
\begin{equation}\label{tau_eq}
    \tau(\widetilde{h}_1 \dots \widetilde{h}_n)=\overline{h}_1 \dots  \overline{h}_n m
\end{equation}
\noindent for some $m \in \overline{M}$. Changing $\overline{h}_n$ if necessary, we assume that $m=1$ so that $\overline{h}_1 \dots  \overline{h}_n \in [\overline{G},\overline{G}]$ and therefore $ (\overline{h}_1 ,\dots, \overline{h}_n ) \in \ker \overline{\psi}_1$.

\medskip

\textbf{Claim 1:} $f$ is well defined, i.e. it does not depend on the choice of the elements $\overline{h}_i$ and moreover ${f(\widetilde{\varphi}_1(\ker \widetilde{\psi}_2)) \subset \overline{\varphi}_1(\ker \overline{\psi}_2)}$.

\textbf{Proof:} We first prove that $f$ does not depend on the choice of $\overline{h}_i$. Suppose that, for each $i=1,\dots,n$, we choose elements $\underline{h}_i \in \overline{H_i}$ satisfying $\widetilde{\lambda}(\widetilde{h}_i)=\overline{\lambda}(\underline{h}_i)$ and $\tau(\widetilde{h}_1\dots\widetilde{h}_n)=\underline{h}_1 \dots \underline{h}_n$. We show that $(\underline{h}_1,\dots,\underline{h}_n)=(\overline{h}_1,\dots,\overline{h}_n)$ in $\ker \overline{\psi}_1 /  \overline{\varphi}_1(\ker \overline{\psi}_2)$. Writing $\underline{h}_i=\overline{h}_i m_i$ for some $m_i \in \overline{M}$, it suffices to prove that $(m_1,\dots,m_n) \in \overline{\varphi}_1(\ker \overline{\psi}_2)$. Since $\overline{h}_1\dots \overline{h}_n =\tau(\widetilde{h}_1 \dots\widetilde{h}_n)=\underline{h}_1 \dots \underline{h}_n$ and the elements $m_i$ are in $\overline{M} \subset Z(\overline{G})$, we obtain $m_1 \dots m_n=1$. As $\overline{M} \subset \bigcap\limits_{i=1}^{n} \overline{H_i}$, multiplying $(m_1,\dots,m_n)$ by $(m_2,m_2^{-1},1,\dots,1)$ (which lies in $\overline{\varphi}_1(\ker \overline{\psi}_2)$ by Lemma \ref{lem:simpl_inters}), we have $(m_1,\dots,m_n)\equiv(m_1 m_2, 1,m_3,\dots , m_n) \pmod{\overline{\varphi}_1(\ker \overline{\psi}_2)}$. Repeating this procedure, we obtain $(m_1,\dots,m_n) \equiv (m_1\dots m_n,\dots,1)=(1,\dots,1) \pmod{\overline{\varphi}_1(\ker \overline{\psi}_2)}$ and therefore $(m_1,\dots,m_n)$ is in $\overline{\varphi}_1(\ker \overline{\psi}_2)$, as desired. 

We now show that $f(\widetilde{\varphi}_1(\ker \widetilde{\psi}_2)) \subset \overline{\varphi}_1(\ker \overline{\psi}_2)$. It suffices to check that $f(\widetilde{\varphi}_1(\ker \widetilde{\psi}_2^{v})) \subset \overline{\varphi}_1(\ker \overline{\psi}_2^{v})$ for any $v \in \Omega_K$. For $i=1,\dots,n$, let $\widetilde{G}=\bigcup\limits_{k=1}^{r_{v,i}} \widetilde{H_i} \widetilde{x}_{i,k} \widetilde{S}_v$ be a double coset decomposition of $\widetilde{G}$ and recall that, by definition, the group $\widetilde{H}_{i,w}$ equals $\widetilde{H_i} \cap \widetilde{x}_{i,k} \widetilde{S}_v \widetilde{x}_{i,k}^{-1}$ if $w \in \Omega_{L_i}$ corresponds to the double coset $\widetilde{H_i} \widetilde{x}_{i,k} \widetilde{S}_v$. Let $\alpha = \bigoplus\limits_{i=1}^{n} \bigoplus\limits_{k=1}^{r_{v,i}} \widetilde{h}_{i,k}  \in \ker \widetilde{\psi}_2^{v}$, where $\widetilde{h}_{i,k} \in \widetilde{H}_{i,w}$ for all possible $i,k$. We thus have 
\begin{equation}\label{assumpt_v0}
    \widetilde{\psi}_2(\alpha)=\prod\limits_{i=1}^{n}\prod\limits_{k=1}^{r_{v,i}} \widetilde{x}_{i,k}^{-1} \widetilde{h}_{i,k} \widetilde{x}_{i,k}  \in [\widetilde{S}_v,\widetilde{S}_v].
\end{equation}

\noindent For any $i=1,\dots,n$ define $\widetilde{h}_i= \prod\limits_{k=1}^{r_{v,i}} \widetilde{h}_{i,k}$. We need to show that $f(\widetilde{h}_1,\dots,\widetilde{h}_n)$ is in $\overline{\varphi}_1(\ker \overline{\psi}_2^v)$.

Set $x_{i,k}:=\widetilde{\lambda}(\widetilde{x}_{i,k}) \in G$ and $h_{i,k}:=\widetilde{\lambda}(\widetilde{h}_{i,k})  \in {H}_{i} \cap x_{i,k} G_v x_{i,k}^{-1}$ for all possible $i,k$. We have $\prod\limits_{i=1}^{n} \prod\limits_{k=1}^{r_{v,i}} x_{i,k}^{-1} h_{i,k} x_{i,k} \in [G_v,G_v]$. Let $\overline{x}_{i,k} \in \overline{G}$ be such that $\overline{\lambda}(\overline{x}_{i,k})=x_{i,k}$ and $\overline{h}_{i,k} \in \overline{H}_{i} \cap \overline{x}_{i,k} \overline{G_v} \overline{x}_{i,k}^{-1}$ satisfying $\overline{\lambda}(\overline{h}_{i,k})=h_{i,k}$. Multiplying one of the $\overline{h}_{1,k}$ by an element of $\overline{M}$ if necessary, we can assure that \begin{equation}\label{assumpt_v}
    \prod\limits_{i=1}^{n} \prod\limits_{k=1}^{r_{v,i}} \overline{x}_{i,k}^{-1} \overline{h}_{i,k} \overline{x}_{i,k} \in [\overline{G_v},\overline{G_v}].
\end{equation}

\noindent In particular, $\alpha':=\bigoplus\limits_{i=1}^{n}\bigoplus\limits_{k=1}^{r_{v,i}} \overline{h}_{i,k}$ is in $ \ker \overline{\psi}_2^{v}$. Defining $\overline{h}_i := \prod\limits_{k=1}^{r_{v,i}} \overline{h}_{i,k}$ for $i=1,\dots,n$, we get $\overline{\varphi_1}(\alpha')=(\overline{h}_1 , \dots, \overline{h}_n )$. We have $\widetilde{\lambda}(\widetilde{h}_i)=\overline{\lambda}(\overline{h}_i)$ by construction and therefore
$$\tau(\widetilde{h}_1 \dots \widetilde{h}_n)=\overline{h}_1 \dots \overline{h}_n m$$
\noindent for some $m \in \overline{M}$. We prove that $m$ is also in $[\overline{G_v},\overline{G_v}]$ so that, by multiplying one of the elements $\overline{h}_{1,k}$ by $m^{-1} \in \overline{M} \cap [\overline{G_v},\overline{G_v}]$ if necessary (note that doing so does not change condition \eqref{assumpt_v}), we obtain $f(\widetilde{h}_1,\dots,\widetilde{h}_n)=(\overline{h}_1,\dots,\overline{h}_n) $. As $(\overline{h}_1,\dots,\overline{h}_n)$ is in $\overline{\varphi}_1(\ker \overline{\psi}_2^{v})$, this proves the claim.

Note that 
$$\prod\limits_{i=1}^{n} \prod\limits_{k=1}^{r_{v,i}}  \widetilde{h}_{i,k}=(\prod\limits_{i=1}^{n} \prod\limits_{k=1}^{r_{v,i}}  \widetilde{h}_{i,k}) (\prod\limits_{i=n}^{1} \prod\limits_{k=r_{v,i}}^{1} \widetilde{x}_{i,k}^{-1} \widetilde{h}_{i,k}^{-1} \widetilde{x}_{i,k}) \widetilde{\psi}_2(\alpha).$$
\noindent Denote $(\prod\limits_{i=1}^{n} \prod\limits_{k=1}^{r_{v,i}}  \widetilde{h}_{i,k}) (\prod\limits_{i=n}^{1} \prod\limits_{k=r_{v,i}}^{1} \widetilde{x}_{i,k}^{-1} \widetilde{h}_{i,k}^{-1} \widetilde{x}_{i,k})$ by $\beta$. Then $\beta \in [\widetilde{G},\widetilde{G}]$ and using an explicit description of $\beta$ as a product of commutators and Lemma \ref{lem:tau_1}\ref{prop_tau2}, we deduce that $\tau(\beta)=\beta'$, where $\beta'=(\prod\limits_{i=1}^{n} \prod\limits_{k=1}^{r_{v,i}}  \overline{h}_{i,k}) (\prod\limits_{i=n}^{1} \prod\limits_{k=r_{v,i}}^{1} \overline{x}_{i,k}^{-1} \overline{h}_{i,k}^{-1} \overline{x}_{i,k})$. Therefore, we have
$$\prod\limits_{i=1}^{n} \overline{h}_{i}=\prod\limits_{i=1}^{n} \prod\limits_{k=1}^{r_{v,i}}  \overline{h}_{i,k} \equiv \beta' = \tau(\beta) \equiv \tau(\prod\limits_{i=1}^{n} \widetilde{h}_{i}) \pmod{[\overline{G_v},\overline{G_v}]},$$
\noindent and thus $m \in [\overline{G_v},\overline{G_v}]$, as desired.
\medskip

\textbf{Claim 2:} $f$ is a homomorphism.

\textbf{Proof:} Let $h=(\widetilde{h}_1 ,\dots,\widetilde{h}_n ),h'=(\widetilde{h}'_1 ,\dots,\widetilde{h}'_n ) \in \ker \widetilde{\psi}_1$ and write $f(h)=(\overline{h}_1 ,\dots, \overline{h}_n )$ and $f(h')=(\overline{h}'_1 ,\dots, \overline{h}'_n )$ for some elements $\overline{h}_i,\overline{h}'_i \in \overline{H_i}$. We have $f(h)f(h')=(\overline{h}_1\overline{h}'_1 ,\dots,\overline{h}_n\overline{h}'_n )$. On the other hand, ${hh'=(\widetilde{h}_1\widetilde{h}'_1 ,\dots,\widetilde{h}_n\widetilde{h}'_n )}$ and
$$\tau(\widetilde{h}_1 \widetilde{h}_1' \dots \widetilde{h}_n \widetilde{h}_n') \equiv \tau((\widetilde{h}_1 \dots \widetilde{h}_n)( \widetilde{h}_1' \dots \widetilde{h}_n')) = (\overline{h}_1\dots \overline{h}_n) (\overline{h}_1' \dots \overline{h}_n') \equiv \overline{h}_1\overline{h}_1'\dots \overline{h}_n\overline{h}_n' \pmod{[\overline{G},\overline{G}]}.$$

\noindent Since $\widetilde{\lambda}(\widetilde{h}_i \widetilde{h}_i')=\overline{\lambda}(\overline{h}_i \overline{h}_i')$ for all $i=1,\dots,n$ and $(\overline{h}_1\dots \overline{h}_n) (\overline{h}_1' \dots \overline{h}_n') \in [\overline{G},\overline{G}]$, by the definition of $f$ it follows that $f(h h')= (\overline{h}_1\overline{h}_1',\dots, \overline{h}_n\overline{h}_n')=f(h)f( h')$.

\medskip

\textbf{Claim 3:} $f$ is surjective.

\textbf{Proof:} For $i=1,\dots,n$, let $\overline{h}_i \in \overline{H_i}$ be such that $\overline{h}_1 \dots \overline{h}_n \in [\overline{G},\overline{G}]$. Take any elements $\widetilde{h}_i \in \widetilde{H_i}$ satisfying $\widetilde{\lambda}(\widetilde{h}_i)=\overline{\lambda}(\overline{h}_i)$. As above, by Lemma \ref{lem:tau_1}\ref{prop_tau1} this implies that there exists $m \in \overline{M}$ such that
$$\tau(\widetilde{h}_1\dots \widetilde{h}_n)=\overline{h}_1 \dots \overline{h}_n m \in [\overline{G},\overline{G}].$$
\noindent Since $\overline{h}_1\dots \overline{h}_n \in [\overline{G},\overline{G}]$, we have $m \in \overline{M} \cap [\overline{G},\overline{G}]$. But $\overline{M} \cap [\overline{G},\overline{G}] = \tau(\widetilde{M} \cap [\widetilde{G},\widetilde{G}])$ by Lemma \ref{lem:tau_1}. Therefore $m = \tau(m')$ for some $m' \in \widetilde{M} \cap [\widetilde{G},\widetilde{G}]$ and thus $(\overline{h}_1,\dots,\overline{h}_n)=f(\widetilde{h}_1,\dots,\widetilde{h}_n m'^{-1})$.

\medskip

\textbf{Claim 4:} $f$ is an isomorphism.

\textbf{Proof:} We have seen that $f$ is surjective. Now we can analogously define a surjective map from $\ker \overline{\psi}_1 /  \overline{\varphi}_1(\ker \overline{\psi}_2)$ to $\ker \widetilde{\psi}_1 / \widetilde{\varphi}_1(\ker \widetilde{\psi}_2)  $. It follows that the finite groups $\ker \widetilde{\psi}_1 / \widetilde{\varphi}_1(\ker \widetilde{\psi}_2) $ and $\ker \overline{\psi}_1 /  \overline{\varphi}_1(\ker \overline{\psi}_2)$ have the same size and so $f$ is an isomorphism.\end{proof}   

Using this theorem, one can also obtain descriptions of the birational invariant $\operatorname{H}^1(K,\Pic \overline{X})$ and the defect of weak approximation $A(T)$ for the multinorm one torus $T$:

\begin{theorem}\label{thm:main_result}
Let $T$ be the multinorm one torus associated to $L$ and let $X$ be a smooth compactification of $T$. In the notation of diagram \eqref{diag:1stobs_defn_generalized}, we have
$$\Sha(T) \cong \ker \widetilde{\psi}_1 / \widetilde{\varphi}_1(\ker \widetilde{\psi}_2),$$
$$\operatorname{H}^1(K,\Pic \overline{X})\cong \ker \widetilde{\psi}_1 / \widetilde{\varphi}_1(\ker \widetilde{\psi}_2^{{nr}}),$$
$$A(T) \cong  \widetilde{\varphi}_1(\ker \widetilde{\psi}_2) / \widetilde{\varphi}_1(\ker \widetilde{\psi}_2^{{nr}}).$$
\end{theorem}

\begin{proof}
The first isomorphism is the statement of Theorem \ref{thm:main_knot} (recall that $\Sha(T)$ is canonically isomorphic to $\mathfrak{K}(L,K)$). The two remaining isomorphisms follow in the same way as in the Hasse norm principle case, see \cite[p. 32--33]{Drak}.
\end{proof}

\begin{remark}\label{rem:finite_time2}

As explained in Remark \ref{rem:finite_time}, all the groups $ \ker \widetilde{\psi}_1, \widetilde{\varphi}_1(\ker \widetilde{\psi}_2)$ and $\widetilde{\varphi}_1(\ker \widetilde{\psi}_2^{nr})$ in Theorem \ref{thm:main_result} can be computed in finite time. To this extent, we assembled a function in GAP \cite{gap} (whose code is available in \cite{macedo_code}) that, given the relevant local and global Galois groups, outputs the obstructions to the multinorm principle and weak approximation for the multinorm one torus of a finite number of extensions by means of Theorem \ref{thm:main_result}.

\end{remark}

We end this section by generalizing Corollary \ref{cor:square_free} and proving that, in many situations, one can actually circumvent the use of generalized representation groups when computing the obstructions to the local-global principles. 

Before we present this result, we need to introduce the notion of focal subgroups. For a moment, let $G$ be any finite group and let $H$ be a subgroup of $G$. The \textit{focal subgroup of $H$ in $G$} is defined as $\Phi^{{G}}({H})=\langle [h,x]  |  h \in {H} \cap x {H} x^{-1}, x \in {G} \rangle$. In \cite[Theorem 2]{DP}, it was proved that $${\varphi}_1(\ker {\psi}_2^{{nr}}) = \Phi^{G}(H) / [H,H]$$
\noindent in the setting of the first obstruction to the Hasse norm principle (case $n=1$). Returning to the multinorm context, this fact promptly implies that, in the notation of diagram \eqref{diag:1stobs_defn_generalized}, we have
\begin{equation}\label{eq:nr_hnp_inclusion}
    (1,\dots,\underbrace{\Phi^{\widetilde{G}}(\widetilde{H_i})}_{i\textrm{-th entry}}, 1,\dots, 1) \subset \widetilde{\varphi}_1(\ker \widetilde{\psi}_2^{nr}).
\end{equation}
\noindent for every $i=1,\dots,n$.

\begin{proposition}\label{prop:sq_free_mid_gp}

Suppose that there exists $j \in \{1,\dots,n\}$ such that, for every prime $p$ dividing $|\hat{\operatorname{H}}^{-3}(G,\Z)|$, $p^2$ does not divide $[L_j:K]$. Then, in the notation of diagram \eqref{diag:1stobs_defn}, we have
$$\Sha(T) \cong \ker {\psi}_1 / {\varphi}_1(\ker {\psi}_2),$$
$$\operatorname{H}^1(K,\Pic \overline{X})\cong \ker {\psi}_1 / {\varphi}_1(\ker {\psi}_2^{{nr}}),$$
$$A(T) \cong  {\varphi}_1(\ker {\psi}_2) / {\varphi}_1(\ker {\psi}_2^{{nr}}).$$
\end{proposition}

\begin{proof}
We prove only that $\operatorname{H}^1(K,\Pic \overline{X})\cong \ker {\psi}_1 / {\varphi}_1(\ker {\psi}_2^{{nr}})$ (the other two isomorphisms can be obtained by a similar argument). Assume, without loss of generality, that $j=1$ and $\widetilde{G}$ is a Schur covering group of $G$ so that $  \widetilde{M}$ is contained in $ [\widetilde{G},\widetilde{G}]$ and $\widetilde{M} \cong \hat{\operatorname{H}}^{-3}(G,\Z)$. We show that the map
\begin{align*}
  \rho \colon\ker \widetilde{\psi}_1 / \widetilde{\varphi}_1(\ker \widetilde{\psi}_2^{nr}) &\longrightarrow \ker {\psi}_1 / {\varphi}_1(\ker {\psi}_2^{nr})\\
  h=(\widetilde{h}_1 ,\dots,\widetilde{h}_n ) &\longmapsto (\widetilde{\lambda}(\widetilde{h}_1) ,\dots,\widetilde{\lambda}(\widetilde{h}_n))
\end{align*}

\noindent is an isomorphism, which proves the desired statement by Theorem \ref{thm:main_result}.

We first verify that $\rho$ is well defined. It is enough to check that $\rho(\widetilde{\varphi}_1(\ker \widetilde{\psi}_2^{v})) \subset {\varphi}_1(\ker {\psi}_2^{v})$ for an unramified place $v$ of $N/K$. Note that if $\widetilde{G}=\bigcup\limits_{k=1}^{r_{v,i}} \widetilde{H_i} \widetilde{x}_{i,k} \widetilde{S}_v$ is a double coset decomposition of $\widetilde{G}$, then ${G}=\bigcup\limits_{k=1}^{r_{v,i}} {H}_i {x}_{i,k} G_v$ is a double coset decomposition of ${G}$, where ${x}_{i,k}=\widetilde{\lambda}(\widetilde{x}_{i,k})$. From this observation, it is straightforward to verify that $\rho(\widetilde{\varphi}_1(\ker \widetilde{\psi}_2^{v})) \subset {\varphi}_1(\ker {\psi}_2^{v})$.

We now prove that $\rho$ is surjective. Suppose that we are given, for $i=1,\dots,n$, elements $h_i \in H_i$ such that $h_1 \dots h_n \in [G,G]$. Since $\widetilde{M} \subset [\widetilde{G},\widetilde{G}]$, any choice of elements $\widetilde{h}_i \in \widetilde{H_i}$ such that $\widetilde{\lambda}(\widetilde{h}_i)=h_i$ will satisfy $\widetilde{h}_1 \dots \widetilde{h}_n \in [\widetilde{G},\widetilde{G}]$ and thus $({h}_1, \dots, {h}_n)=\rho(\widetilde{h}_1 ,\dots ,\widetilde{h}_n)$.

We finally show that $\rho$ is injective. Suppose that $(h_1,\dots,h_n)=\rho(h) \in {\varphi}_1(\ker {\psi}_2^{v})$ for some unramified place $v$ of $N/K$. Write $h_i=\varphi_1(\bigoplus\limits_{k=1}^{r_{v,i}} h_{i,k})$ for some elements $h_{i,k} \in H_i \cap x_{i,k} G_v x_{i,k}^{-1}$. As $\rho(h) \in {\varphi}_1(\ker {\psi}_2^{v})$, we obtain $\prod\limits_{i=1}^{n}\prod\limits_{k=1}^{r_{v,i}} x_{i,k}^{-1} h_{i,k} x_{i,k}=1$. Picking elements $\widetilde{h}_{i,k}\in\widetilde{\lambda}^{-1}(h_{i,k})$ and $\widetilde{x}_{i,k}\in\widetilde{\lambda}^{-1}(x_{i,k})$ for all possible $i,k$, we obtain $\prod\limits_{i=1}^{n}\prod\limits_{k=1}^{r_{v,i}} \widetilde{x}_{i,k}^{-1} \widetilde{h}_{i,k} \widetilde{x}_{i,k}=m$ for some $m \in \widetilde{M} = \ker \widetilde{\lambda}$. As $m \in Z(\widetilde{G}) \cap \bigcap\limits_{i=1}^{n} \widetilde{H_i}$, we have $(\widetilde{h}_1 m^{-1},\widetilde{h}_2, ,\dots,\widetilde{h}_n) \in \widetilde{\varphi}_1(\ker \widetilde{\psi}_2^{nr})$. Therefore, in order to prove that $h \in \widetilde{\varphi}_1(\ker \widetilde{\psi}_2^{nr})$ it suffices to show that $(m^{-1},1,\dots,1) \in \widetilde{\varphi}_1(\ker \widetilde{\psi}_2^{nr})$. We prove that $m \in \Phi^{\widetilde{G}}(\widetilde{H_1})$, which completes the proof by \eqref{eq:nr_hnp_inclusion}.

\medskip
\textbf{Claim:} If $p^2$ does not divide $[L_1:K]$ for every prime $p$ dividing $|\widetilde{M}|$, then $\widetilde{M} \subset \Phi^{\widetilde{G}}(\widetilde{H_1})$.

\textbf{Proof:} We show that $\widetilde{M}_{(p)} \subset \Phi^{\widetilde{G}}(\widetilde{H_1})$. We have $[L_1:K]=[G:H_1]$ and therefore $[G_p:(H_1)_{p}] = [\widetilde{G}_p :(\widetilde{H_1})_{p}]=1$ or $p$. In any case, $(\widetilde{H_1})_p \trianglelefteq \widetilde{G}_p$ and we can write $\widetilde{G}_p = \langle x_p \rangle . (\widetilde{H_1})_p $ for some $x_p \in \widetilde{G}_p$. Since $\widetilde{M}_{(p)} \subset \widetilde{G}_p \cap [\widetilde{G},\widetilde{G}] \cap Z(\widetilde{G})$ and $\widetilde{G}_p \cap [\widetilde{G},\widetilde{G}] \cap Z(\widetilde{G}) \subset [\widetilde{G}_p,\widetilde{G}_p]$ (this last inclusion follows from properties of the transfer map, e.g. \cite[Lemma 5.5]{Isaacs}), we have $\widetilde{M}_{(p)}  \subset [\widetilde{G}_p ,\widetilde{G}_p ]$ and so it suffices to prove that $[\widetilde{G}_p ,\widetilde{G}_p ] \subset \Phi^{\widetilde{G}}(\widetilde{H_1})$. Let $z=[x_p^a h_1, x_p^b h_1']$ for some $a,b \in \Z$ and $h_1,h_1' \in (\widetilde{H_1})_p$. Using the commutator properties, we have $z=[x_p^a,h_1']^{h_1}[h_1,h_1'][h_1,x_p^b]^{h_1'}$. As $(\widetilde{H_1})_p \trianglelefteq \widetilde{G}_p $ and $\Phi^{\widetilde{G}}(\widetilde{H_1}) \trianglelefteq \widetilde{H_1}$, it follows that each one of the commutators above is in $\Phi^{\widetilde{G}}(\widetilde{H_1})$.\qedhere
\end{proof}

As a consequence we obtain the following result, which can be thought of as an analog of {\cite[Corollary 1]{DP}} for the birational invariant $\operatorname{H}^1(K,\Pic \overline{X})$.

\begin{corollary}
Let $L/K$ be an extension of global fields and suppose that $[L:K]$ is square-free. Let $X$ be a smooth compactification of the norm one torus $R^1_{L/K} \bbG_m$. Then $$\operatorname{H}^1(K,\Pic \overline{X}) \cong \frac{H \cap [G,G]}{\Phi^{G}(H)}.$$
\end{corollary}

\begin{proof}
The conditions of Proposition \ref{prop:sq_free_mid_gp} are satisfied and hence $\operatorname{H}^1(K,\Pic \overline{X}) \cong  \ker {\psi}_1 / {\varphi}_1(\ker {\psi}_2^{nr})$. The result then follows from \cite[Theorem 2]{DP}.
\end{proof}

\section{Applications}\label{sec:applications}

In this section we employ the techniques developed so far in order to analyze the multinorm principle or weak approximation for the multinorm one torus in three different situations. Namely, we extend results of Bayer-Fluckiger--Lee--Parimala \cite{eva}, Demarche--Wei \cite{demarche} and Pollio \cite{pollio}. The notation used throughout this section is as in Sections \ref{sec:1st_obs} and \ref{sec:gen_gps}. Additionally, we will make use of the norm one torus $S=R^1_{F/K} \bbG_m$ of the extension $F=\bigcap\limits_{i=1}^{n} L_i$ and we let $Y$ denote a smooth compactification of $S$. We start by establishing a few auxiliary lemmas to be used in later applications.

\subsection{Preliminary results}

\begin{lemma}\label{lem:incl_unr}
In the notation of diagram \eqref{diag:1stobs_defn_generalized}, we have $$\widetilde{\varphi}_1(\ker \widetilde{\psi}_2^{nr}) \subseteq \{ (h_1 \widetilde{H_1}' , \dots , h_n \widetilde{H_n}') \in \ker \widetilde{\psi}_1  |  h_1\dots h_n \in \Phi^{\widetilde{G}}(\widetilde{H}) \}.$$
\end{lemma}

A proof of this lemma can be obtained by following the same strategy as in the proof of the analogous result for the Hasse norm principle (case $n=1$) in \cite[Theorem 2]{DP}. Nonetheless, as the details are slightly intricate, we include a proof here for the benefit of the reader.

\begin{proof}
Since $\widetilde{\varphi}_1(\ker \widetilde{\psi}_2^{nr}) = \prod\limits_{\substack{v \in \Omega_K \\ v \text{ unramified}}} \widetilde{\varphi}_1(\ker \widetilde{\psi}_2^{ v })$, it suffices to prove that
$$\widetilde{\varphi}_1(\ker \widetilde{\psi}_2^{v}) \subseteq \{ (h_1 \widetilde{H_1}' , \dots , h_n \widetilde{H_n}') | h_1\dots h_n \in \Phi^{\widetilde{G}}(\widetilde{H}) \} $$
\noindent for any unramified place $v$ of $N/K$. Let $\alpha \in \ker \widetilde{\psi}_2^{ v }$ and fix a double coset decomposition $\widetilde{G}=\bigcup\limits_{k=1}^{r_{v,i}} \widetilde{H_i} \widetilde{x}_{i,k} \widetilde{S}_v$. Write $\widetilde{S}_v=\langle g \rangle$ and $\alpha = \bigoplus\limits_{i=1}^{n} \bigoplus\limits_{k=1}^{r_{v,i}} \widetilde{h}_{i,k} $ for some $g \in \widetilde{G}$, $\widetilde{h}_{i,k} = \widetilde{x}_{i,k} g^{e_{i,k}} \widetilde{x}_{i,k}^{-1} \in \widetilde{H_i} \cap \widetilde{x}_{i,k} \langle g \rangle \widetilde{x}_{i,k}^{-1}$ and some $e_{i,k} \in \Z$. By hypothesis, we have
$1=\widetilde{\psi}_2(\alpha)=g^{\sum_{i,k} e_{i,k}}$ and therefore
$$\sum\limits_{i,k} e_{i,k} \equiv 0 \pmod{m},$$
\noindent where $m$ is the order of $g$. Since $g^m=1$, by changing some of the $e_{i,k}$ if necessary, we can (and do) assume that \begin{equation}\label{eq:sum=0}
    \sum\limits_{i,k} e_{i,k} = 0.
\end{equation}

Letting ${h_i} =\prod\limits_{k=1}^{r_{v,i}}  \widetilde{h}_{i,k} $ for any $1 \leq i \leq n$, we have $\widetilde{\varphi}_1(\alpha)=(h_1 \widetilde{H_1},\dots,h_n \widetilde{H_n}) \in \ker \widetilde{\psi}_1$. We prove that  
$$\prod\limits_{i=1}^{n} {h_i} =\prod\limits_{i=1}^{n}(\prod\limits_{k=1}^{r_{v,i}}  \widetilde{h}_{i,k})=\prod\limits_{i=1}^{n} (\prod\limits_{k=1}^{r_{v,i}} \widetilde{x}_{i,k} g^{e_{i,k}} \widetilde{x}_{i,k}^{-1}) \in \Phi^{\widetilde{G}}(\widetilde{H})$$
\noindent by induction on $s:=\sum\limits_{i=1}^{n} r_{v,i}$. The case $s=1$ is trivial and the case $s=2$ is solved in \cite[p. 308]{DP}. Now let $s > 2$ and set $d=\gcd(e_{i,k} | 1 \leq i \leq n, 1 \leq k \leq r_{v,i})$ and $f_{i,k}=\frac{e_{i,k}}{d}$. It follows that $\gcd(f_{i,k} |  1 \leq i \leq n, 1 \leq k \leq r_{v,i})=1$ and, since $\sum\limits_{i,k} f_{i,k} = 0$ by \eqref{eq:sum=0}, we have ${\gcd(f_{i,k} |  1 \leq i \leq n, 1 \leq k \leq r_{v,i} \textrm{ and } (i,k) \neq (n,r_{v,n}))}=1$. Hence there exist $a_{i,k} \in \Z$ such that $\sum\limits_{\substack{i,k\\(i,k) \neq (n,r_{v,n})}} f_{i,k} a_{i,k} = 1$. Consider the element
$$\beta=\Big(\bigoplus\limits_{\substack{i,k\\(i,k) \neq (n,r_{v,n})}} \widetilde{x}_{i,k} g^{e_{i,k}f_{n,r_{v,n}}a_{i,k}} \widetilde{x}_{i,k}^{-1}\Big) \oplus \widetilde{x}_{n,r_{v,n}} g^{-e_{n,r_{v,n}}} \widetilde{x}_{n,r_{v,n}}^{-1} \in \bigoplus\limits_{i=1}^{n}\Big( \bigoplus\limits_{k=1}^{r_{v,i}} {\widetilde{H}_{i,w} }\Big) .$$
\noindent Since 
$e_{i,k} f_{n,r_{v,n}}=e_{n,r_{v,n}} f_{i,k},$
\noindent we have
$$\widetilde{\psi}_2(\beta)=g^{\Big(\sum\limits_{\substack{i,k\\(i,k) \neq (n,r_{v,n})}} e_{i,k}f_{n,r_{v,n}}a_{i,k}\Big)-e_{n,r_{v,n}}}=g^{\Big(\sum\limits_{\substack{i,k\\(i,k) \neq (n,r_{v,n})}} e_{n,r_{v,n}}f_{i,k}a_{i,k}\Big)-e_{n,r_{v,n}}}=1$$
\noindent and so $\beta \in \ker \widetilde{\psi}_2^{v}$. 

Additionally, if $\widetilde{\varphi}_1(\beta) = (\widetilde{h}_1,\dots,\widetilde{h}_n)$, we have 

\begin{equation}
    \begin{split}
        \prod\limits_{i=1}^{n} \widetilde{h}_i & = \left(\prod\limits_{\substack{i,k\\(i,k) \neq (n,r_{v,n})}} \widetilde{x}_{i,k} g^{e_{i,k}f_{n,r_{v,n}}a_{i,k}} \widetilde{x}_{i,k}^{-1} \right)  \widetilde{x}_{n,r_{v,n}} g^{-e_{n,r_{v,n}}} \widetilde{x}_{n,r_{v,n}}^{-1} =  \\
        & = \left(\prod\limits_{\substack{i,k\\(i,k) \neq (n,r_{v,n})}} \widetilde{x}_{i,k} g^{e_{i,k}f_{n,r_{v,n}}a_{i,k}} \widetilde{x}_{i,k}^{-1} \right)  \widetilde{x}_{n,r_{v,n}} g^{-e_{n,r_{v,n}} \sum\limits_{\substack{i,k\\(i,k) \neq (n,r_{v,n})}} f_{i,k} a_{i,k}} \widetilde{x}_{n,r_{v,n}}^{-1} \equiv \\
        & \equiv \left(\prod\limits_{\substack{i,k\\(i,k) \neq (n,r_{v,n})}} \widetilde{x}_{i,k} g^{e_{i,k}f_{n,r_{v,n}}a_{i,k}} \widetilde{x}_{i,k}^{-1} \widetilde{x}_{n,r_{v,n}} g^{-e_{i,k}f_{n,r_{v,n}}a_{i,k}} \widetilde{x}_{n,r_{v,n}}^{-1} \right) \pmod{[\widetilde{H},\widetilde{H}]}
    \end{split}
\end{equation}
\noindent since the elements $\widetilde{x}_{i,k} g^{e_{i,k}} \widetilde{x}_{i,k}^{-1}$ (for all possible $i,k$) are in $\widetilde{H}$. Arguing similarly to the case $s=2$ (see \cite[p. 308]{DP}), we deduce that $\prod\limits_{i=1}^{n} \widetilde{h}_i \in \Phi^{\widetilde{G}}(\widetilde{H})$. Finally, consider the element $$\alpha'=\alpha \beta=\bigoplus\limits_{\substack{i,k\\(i,k) \neq (n,r_{v,n})}} \widetilde{x}_{i,k} g^{e_{i,k}(1+f_{n,r_{v,n}}a_{i,k})} \widetilde{x}_{i,k}^{-1} \in \bigoplus\limits_{i=1}^{n}\Big( \bigoplus\limits_{k=1}^{r_{v,i}} {\widetilde{H}_{i,w} }\Big).$$
\noindent It is clear that $\alpha' \in \ker \widetilde{\psi}_2^{v}$. By the induction hypothesis, if $\widetilde{\varphi}_1(\alpha')=(\widehat{h}_1,\dots,\widehat{h}_n)$ we have $\widehat{h}_1 \dots \widehat{h}_n \in \Phi^{\widetilde{G}}(\widetilde{H})$. Since $\widehat{h}_i \equiv h_i \widetilde{h}_i \pmod{[\widetilde{H},\widetilde{H}]}$ for all $i=1,\dots,n$, we conclude that ${h_1} \dots {h_n} \in \Phi^{\widetilde{G}}(\widetilde{H})$ as well.
\end{proof}

\begin{lemma}\label{lem:surject_int_Galois}
\begin{enumerate}[leftmargin=*,label=(\roman{*})]
    \item\label{lem:surject_int_Galois1} There exists a surjection $f:\operatorname{H}^1(K,\Pic \overline{X}) \xrightarrow{} \operatorname{H}^1(K,\Pic \overline{Y})$. If in addition 
$$\widetilde{\varphi}_1(\ker \widetilde{\psi}_2^{nr}) \supseteq \{ (h_1 \widetilde{H_1}' , \dots , h_n \widetilde{H_n}') | h_1\dots h_n \in \Phi^{\widetilde{G}}(\widetilde{H}) \}$$
\noindent (in the notation of diagram \eqref{diag:1stobs_defn_generalized}), then $f$ is an isomorphism. 
    \item\label{lem:surject_int_Galois2} If $F/K$ is Galois, $f$ induces a surjection $\Sha(T) \twoheadrightarrow \Sha(S)$.
\end{enumerate}

\end{lemma}

\begin{proof}

Consider the analog of diagram \eqref{diag:1stobs_defn_generalized} for the extension $F/K$ (note that this is the fixed field of the group ${H}$ inside $N/K$):
\begin{equation}\label{diag:F/K_v0}
\xymatrix{
\widetilde{H}^{\textrm{ab}}  \ar[r]^{\widehat{\psi}_1} & \widetilde{G}^{\textrm{ab}}\\
\bigoplus\limits_{v \in \Omega_K} ( \bigoplus\limits_{w|v} {\widetilde{H}_{w}^{\textrm{ab}} })  \ar[r]^{\ \ \ \ \ \ \ \widehat{\psi}_2} \ar[u]^{\widehat{\varphi}_1}&\bigoplus\limits_{v \in \Omega_K}{\widetilde{S}_v^{\textrm{ab}} }\ar[u]_{\widehat{\varphi}_2}
}
\end{equation}

\noindent Here all the maps with the $\widehat{\phantom{a}}$ notation are defined as in diagram \eqref{diag:1stobs_defn_generalized} with respect to the extension $F/K$. Now define
\begin{align*}
  f \colon \ker \widetilde{\psi}_1 / \widetilde{\varphi}_1(\ker \widetilde{\psi}_2^{nr}) &\longrightarrow \ker \widehat{\psi}_1 / \widehat{\varphi}_1(\ker \widehat{\psi}_2^{nr}) \\
  (\widetilde{h}_1 \widetilde{H_1}',\dots,\widetilde{h}_n \widetilde{H_n}')\widetilde{\varphi}_1(\ker \widetilde{\psi}_2^{nr}) &\longmapsto (\widetilde{h}_1 \dots\widetilde{h}_n[\widetilde{H},\widetilde{H}]) \widehat{\varphi}_1(\ker \widehat{\psi}_2^{nr})
\end{align*}
\noindent Since $\widehat{\varphi}_1(\ker \widehat{\psi}_2^{nr})=\Phi^{\widetilde{G}}(\widetilde{H})/[\widetilde{H},\widetilde{H}]$ (see \cite[Theorem 2]{DP}), the map $f$ is well defined by Lemma \ref{lem:incl_unr}. Additionally, as the target group is abelian, it is easy to check that $f$ is a homomorphism and surjective. By Theorem \ref{thm:main_result} we have $\operatorname{H}^1(K,\Pic \overline{X}) \cong \ker \widetilde{\psi}_1 / \widetilde{\varphi}_1(\ker \widetilde{\psi}_2^{nr})$ and $\operatorname{H}^1(K,\Pic \overline{Y}) \cong \ker \widehat{\psi}_1 / \widehat{\varphi}_1(\ker \widehat{\psi}_2^{nr})$. The statement in the first sentence follows. Finally, if we assume $\widetilde{\varphi}_1(\ker \widetilde{\psi}_2^{nr}) \supseteq \{ (h_1 \widetilde{H_1}' , \dots , h_n \widetilde{H_n}') | h_1\dots h_n \in \Phi^{\widetilde{G}}(\widetilde{H}) \}$, then it is clear that $f$ is injective. 

We now prove \ref{lem:surject_int_Galois2}. By Theorem \ref{thm:main_result}, it is enough to show that $f( \widetilde{\varphi}_1(\ker \widetilde{\psi}_2)) \subset \widehat{\varphi}_1(\ker \widehat{\psi}_2)$. Since $\widetilde{\varphi}_1(\ker \widetilde{\psi}_2)=\prod\limits_{v \in \Omega_K} \widetilde{\varphi}_1(\ker \widetilde{\psi}_2^{v})$, it suffices to verify $f(\widetilde{\varphi}_1(\ker \widetilde{\psi}_2^{v}))\subset \widehat{\varphi}_1(\ker \widehat{\psi}_2)$ for all $v \in \Omega_K$. Let $\alpha \in \ker \widetilde{\psi}_2^{ v }$ and write $\alpha = \bigoplus\limits_{i=1}^{n} \bigoplus\limits_{k=1}^{r_{v,i}} \widetilde{h}_{i,k} $ for some  $\widetilde{h}_{i,k} \in \widetilde{H_i} \cap \widetilde{x}_{i,k} \widetilde{S}_v \widetilde{x}_{i,k}^{-1}$. Hence, we obtain $\widetilde{\varphi}_1(\alpha)=(\widetilde{h}_1,\dots,\widetilde{h}_n)$, where $\widetilde{h}_i=\prod\limits_{k=1}^{r_{v,i}} \widetilde{h}_{i,k}$, and we wish to show that $\prod\limits_{i=1}^{n} \widetilde{h}_i \in \widehat{\varphi}_1(\ker \widehat{\psi}_2)$. Since $F/K$ is Galois, $\widetilde{H}$ is a normal subgroup of $ \widetilde{G}$ and thus $\Phi^{\widetilde{G}}(\widetilde{H})=[\widetilde{H},\widetilde{G}]$. In this way, we have $$\prod\limits_{i=1}^{n} \widetilde{h}_i = \prod\limits_{i=1}^{n} \prod\limits_{k=1}^{r_{v,i}}  \widetilde{h}_{i,k}  \equiv \prod\limits_{i=1}^{n} \prod\limits_{k=1}^{r_{v,i}} \widetilde{x}_{i,k}^{-1} \widetilde{h}_{i,k} \widetilde{x}_{i,k} = \widetilde{\psi}_2(\alpha) \pmod{\Phi^{\widetilde{G}}(\widetilde{H})}.$$

\noindent As $\Phi^{\widetilde{G}}(\widetilde{H})/[\widetilde{H},\widetilde{H}]=\widehat{\varphi}_1(\ker \widehat{\psi}_2^{nr})$, it suffices to prove that $\widetilde{\psi}_2(\alpha) \in \widehat{\varphi}_1(\ker \widehat{\psi}_2^{v})$. For this, let $\widetilde{G}=\bigcup\limits_{j=1}^{r} \widetilde{H} \widetilde{y}_j \widetilde{S}_v$ be a double coset decomposition and suppose, without loss of generality, that $\widetilde{y}_{j_0}=1$ for some index $1 \leq j_0 \leq r$ corresponding to a place $w_0$ of $F$ via Lemma \ref{lem1DP}. Therefore, we obtain $\widetilde{\psi}_2(\alpha)=\prod\limits_{i=1}^{n} \prod\limits_{k=1}^{r_{v,i}} \widetilde{x}_{i,k}^{-1} \widetilde{h}_{i,k} \widetilde{x}_{i,k} \in \widetilde{H} \cap \widetilde{S}_v=\widetilde{H}_{w_0}$ since $\widetilde{x}_{i,k}^{-1} \widetilde{h}_{i,k} \widetilde{x}_{i,k} \in \widetilde{H}$ for all possible $i,k$. In this way, if $\beta \in \bigoplus\limits_{v \in \Omega_K} ( \bigoplus\limits_{w|v} {\widetilde{H}_{w}^{\textrm{ab}} }) $ is the vector with the $(v,w_0)$-th entry equal to $\widetilde{\psi}_2(\alpha)$ and all other entries equal to $1$, we have $\widehat{\psi}_2(\beta)=\widetilde{\psi}_2(\alpha) \in [\widetilde{S}_v,\widetilde{S}_v]$ (as $\alpha \in \ker \widetilde{\psi}_2^{ v }$) and so $\widetilde{\psi}_2(\alpha) =\widehat{\varphi}_1(\beta) \in \widehat{\varphi}_1(\ker \widehat{\psi}_2^{v}) $.
\end{proof}

\subsection{Multinorm principle for linearly disjoint extensions}

\hspace{1pt}

In this subsection we prove a theorem similar to the main result of \cite{demarche}, but with a slightly different hypothesis (and in some cases more general, see Remark \ref{rem:demarche_different} below).

\begin{theorem}\label{thm:demarch_wei_thm}
For any non-empty subset $I \subset \{1,\dots,n\}$, let $L_I \subseteq N$ be the compositum of the fields $L_i$ $(i \in I)$ and let $E_{I}$ be the Galois closure of $L_I/K$. Suppose that there exist indices $i_0,j_0 \in \{1,\dots,n\}$ such that, for every $1 \leq i \leq n$, there is a partition $I_i \sqcup J_i = \{1,\dots,n\}$ with $i_0 \in I_i,j_0 \in J_i$ and $E_{I_i} \cap E_{J_i} \subseteq L_i$. Then
$$ \operatorname{H}^1(K,\Pic \overline{X}) \cong \operatorname{H}^1(K,\Pic \overline{Y}). $$

\end{theorem}

\begin{proof}
If $n=1$ there is nothing to show, so assume $n \geq 2$. By Lemma \ref{lem:surject_int_Galois}\ref{lem:surject_int_Galois1} it suffices to prove that $$\widetilde{\varphi}_1(\ker \widetilde{\psi}_2^{nr}) \supseteq \{ (h_1 \widetilde{H_1}' , \dots , h_n \widetilde{H_n}') | h_1\dots h_n \in \Phi^{\widetilde{G}}(\widetilde{H}) \}.$$

\noindent Let $\alpha=(h_1 \widetilde{H_1}' , \dots , h_n \widetilde{H_n}')$ be such that $h_1 \dots h_n \in \Phi^{\widetilde{G}}(\widetilde{H})$. Renaming the fields $L_i$ if necessary, we assume that $i_0 =1$ and $j_0=2$. Denoting $B_{I_i}=\Gal(N/E_{I_i}), B_{J_i}=\Gal(N/E_{J_i})$ for all $1 \leq i \leq n$, the hypothesis $E_{I_i} \cap E_{J_i} \subseteq L_i$ is equivalent to $ B_{I_i}B_{J_i} \supseteq H_i$ and thus
\begin{equation}\label{eq:inc_hyp}
    \widetilde{H_i} \subseteq  \widetilde{B_{I_i}} \widetilde{B_{J_i}}
\end{equation}
\noindent with $1 \in I_i$, $2 \in J_i$ and $i \in I_i$ or $J_i$. If $n \geq 3$, this implies that for any $3 \leq i \leq n$ we can decompose $h_i = h_{1,i} h_{2,i}$ for some $h_{1,i} \in \widetilde{H_1} \cap \widetilde{H_i}$ and $h_{2,i} \in \widetilde{H_2} \cap \widetilde{H_i}$. Using Lemma \ref{lem:simpl_inters} as done in Claim 1 of the proof of Theorem \ref{thm:main_knot}, we obtain $$\alpha \equiv ((\prod\limits_{3 \leq i \leq n} h_{1,i}) h_1,(\prod\limits_{3 \leq i \leq n} h_{2,i}) h_2,1,\dots,1)$$
\noindent modulo $\widetilde{\varphi}_1(\ker \widetilde{\psi}_2^{nr})$. We can thus assume $\alpha$ to be of the form $(h_1',h_2',1,\dots,1)$ for some $h_1' \in \widetilde{H_1},h_2' \in \widetilde{H_2}$ such that $h_1' h_2' \in  \Phi^{\widetilde{G}}(\widetilde{H})$. Note that \eqref{eq:inc_hyp} implies that $\widetilde{H} =\langle  \widetilde{H_i} \rangle  \subset \widetilde{B_1} \widetilde{B_2}$, where $B_1=\Gal(N/E_{\{1\}})$ and $B_2=\Gal(N/E_{\{2\}})$. It thus follows that $\Phi^{\widetilde{G}}(\widetilde{H}) \subset \Phi^{\widetilde{G}}(\widetilde{B_1}\widetilde{B_2})=\Phi^{\widetilde{G}}(\widetilde{B_1})\Phi^{\widetilde{G}}(\widetilde{B_2}) $ and so $h_1' h_2' \in \Phi^{\widetilde{G}}(\widetilde{B_1}) \Phi^{\widetilde{G}}(\widetilde{B_2})$. Since $\Phi^{\widetilde{G}}(\widetilde{B_i}) \subset \Phi^{\widetilde{G}}(\widetilde{H_i})$ and recalling that $$(1,\dots,\underbrace{\Phi^{\widetilde{G}}(\widetilde{H_i})}_{i\textrm{-th entry}}, 1,\dots, 1) \subset \widetilde{\varphi}_1(\ker \widetilde{\psi}_2^{nr})$$
\noindent (see \eqref{eq:nr_hnp_inclusion} in Section \ref{sec:gen_gps}), we can multiply $h_1'$ and $h_2'$ by elements of $\widetilde{\varphi}_1(\ker \widetilde{\psi}_2^{nr})$ to attain $\alpha \equiv (h_1'',h_2'',1,\dots,1) \pmod{\widetilde{\varphi}_1(\ker \widetilde{\psi}_2^{nr})}$ for some $h_1'' \in \widetilde{H_1},h_2'' \in \widetilde{H_2}$ such that $h_1'' h_2''=1$. Thus $h_2''=h_1''^{-1}$ and $\alpha = (h_1'',h_1''^{-1},1,\dots,1)$, which by Lemma \ref{lem:simpl_inters} is in $ \widetilde{\varphi}_1(\ker \widetilde{\psi}_2^{nr})$, as desired. \end{proof}

\begin{remark}\label{rem:demarche_different}
It is easy to see that if there exists a partition $I \sqcup J=\{1,\dots,n\}$ such that $E_I \cap E_J=F$ (the assumption in \cite[Theorem 6]{demarche} when $F_i=E_I$ and $F_j=E_J$ for every $i \in I, j \in J$), the conditions of Theorem \ref{thm:demarch_wei_thm} are satisfied. Therefore, our theorem applies to all the cases described in \cite[Example 9(i)--(iii)]{demarche}. Moreover, our hypothesis applies for $n$-tuples of fields for which the assumptions in \cite[Theorem 6]{demarche} might fail. For example, let $L=(\Q(\sqrt{2},\sqrt{3}),\Q(\sqrt{2},\sqrt{5}),\Q(\sqrt{3},\sqrt{5}))$. It is easy to see that the assumptions of Theorem \ref{thm:demarch_wei_thm} are satisfied, but \cite[Theorem 6]{demarche} does not apply to this tuple of fields. Indeed, Demarche and Wei's hypothesis imply that there is a partition $I \sqcup J=\{1,\dots,n\}$ such that $L_I \cap L_J=F$, which does not exist in the example above. 
\end{remark}

As consequence of Theorem \ref{thm:demarch_wei_thm} we also obtain versions of \cite[Corollaries 7 and 8]{demarche}:

\begin{corollary}\label{cor:dem1}
Let $c\in K^*$. Assume the hypothesis of Theorem \ref{thm:demarch_wei_thm} and suppose that the $K$-variety $ N_{F/K}(\Xi)=c$ satisfies weak approximation. Then the multinorm equation $\prod\limits_{i=1}^{n} N_{L_i/K}(\Xi_i)=c$ satisfies weak approximation if and only if it has a $K$-point.
\end{corollary}

\begin{corollary}\label{cor:dem2}
Assume the hypothesis of Theorem \ref{thm:demarch_wei_thm} and suppose that the Hasse principle and weak approximation hold for all norm equations $ N_{F/K}(\Xi)=c$, $c \in K^*$. Then the Hasse principle and weak approximation hold for all multinorm equations $\prod\limits_{i=1}^{n} N_{L_i/K}(\Xi_i)=c$.
\end{corollary}

\subsection{Multinorm principle and weak approximation for abelian extensions}

\hspace{1pt}

In this subsection we generalize the main theorem of \cite{pollio} to $n$ abelian extensions under the conditions of Theorem \ref{thm:demarch_wei_thm}.

\begin{theorem}\label{thm:pollio}
Let $L=(L_1,\dots,L_n)$ be an $n$-tuple of abelian extensions of $K$ and suppose that the conditions of Theorem \ref{thm:demarch_wei_thm} are satisfied for $L$. Then
$$\Sha(T) \cong \Sha(S),$$ 
$$A(T) \cong A(S).$$
\end{theorem}

\begin{proof}
Note that if $A(T) \cong A(S)$, then by Theorem \ref{thm:demarch_wei_thm} and Voskresenski\u{\i}'s exact sequence \eqref{eq:Vosk} we deduce that $|\Sha(T)|=|\Sha(S)|$. Since $\Sha(T)$ surjects onto $\Sha(S)$ by Lemma \ref{lem:surject_int_Galois}\ref{lem:surject_int_Galois2}, we conclude that $\Sha(T) \cong \Sha(S)$. Therefore, it is enough to prove that $A(T) \cong A(S)$.

Let us again consider the analog of diagram \eqref{diag:1stobs_defn_generalized} for the extension $F/K$:
\begin{equation}\label{diag:F/K_v2}
\xymatrix{
\widetilde{H}^{\textrm{ab}}  \ar[r]^{\widehat{\psi}_1} & \widetilde{G}^{\textrm{ab}}\\
\bigoplus\limits_{v \in \Omega_K} ( \bigoplus\limits_{w|v} {\widetilde{H}_{w}^{\textrm{ab}} })  \ar[r]^{\ \ \ \ \ \ \ \widehat{\psi}_2} \ar[u]^{\widehat{\varphi}_1}&\bigoplus\limits_{v \in \Omega_K}{\widetilde{S}_v^{\textrm{ab}} }\ar[u]_{\widehat{\varphi}_2}
}
\end{equation}

\noindent As before, in this diagram all the maps with the $\widehat{\phantom{a}}$ superscript are defined as in diagram \eqref{diag:1stobs_defn_generalized} with respect to $F/K$. By Theorem \ref{thm:main_result}, we have $A(T) \cong \widetilde{\varphi}_1(\ker \widetilde{\psi}_2)/ \widetilde{\varphi}_1(\ker \widetilde{\psi}_2^{nr})$ (in the notation of diagram \eqref{diag:1stobs_defn_generalized}) and $A(S) \cong  \widehat{\varphi}_1(\ker \widehat{\psi}_2) / \widehat{\varphi}_1(\ker \widehat{\psi}_2^{nr})$ (in the notation of diagram \eqref{diag:F/K_v2}). Therefore it suffices to show that $ \widetilde{\varphi}_1(\ker \widetilde{\psi}_2) /  \widetilde{\varphi}_1(\ker \widetilde{\psi}_2^{nr})$ is isomorphic to $  \widehat{\varphi}_1(\ker \widehat{\psi}_2) / \widehat{\varphi}_1(\ker \widehat{\psi}_2^{nr})$. For this, we again consider the natural map
\begin{align*}
  f \colon\widetilde{\varphi}_1(\ker \widetilde{\psi}_2) /  \widetilde{\varphi}_1(\ker \widetilde{\psi}_2^{nr}) &\longrightarrow \widehat{\varphi}_1(\ker \widehat{\psi}_2) / \widehat{\varphi}_1(\ker \widehat{\psi}_2^{nr}) \\
  (\widetilde{h}_1 \widetilde{H_1}',\dots,\widetilde{h}_n \widetilde{H_n}')\widetilde{\varphi}_1(\ker \widetilde{\psi}_2^{nr}) &\longmapsto (\widetilde{h}_1 \dots \widetilde{h}_n [\widetilde{H},\widetilde{H}])\widehat{\varphi}_1(\ker \widehat{\psi}_2^{nr}) 
\end{align*}

\noindent In the proof of Lemma \ref{lem:surject_int_Galois}\ref{lem:surject_int_Galois2} it was shown that $f(\widetilde{\varphi}_1(\ker \widetilde{\psi}_2))\subset \widehat{\varphi}_1(\ker \widehat{\psi}_2)$. Additionally, recalling that $\widehat{\varphi}_1(\ker \widehat{\psi}_2^{nr}) =\Phi^{\widetilde{G}}(\widetilde{H})/[\widetilde{H},\widetilde{H}]$ by \cite[Theorem 2]{DP}, we have $f(\widetilde{\varphi}_1(\ker \widetilde{\psi}_2^{nr}))= \widehat{\varphi}_1(\ker \widehat{\psi}_2^{nr})$ by Lemma \ref{lem:incl_unr} and the proof of Theorem \ref{thm:demarch_wei_thm}. This shows that $f$ is well defined and injective. 

Finally, let us check that $f$ is surjective. Fix a place $v$ of $K$ and a double coset decomposition $\widetilde{G}=\bigcup\limits_{j=1}^{r}\widetilde{H} \widetilde{y}_j \widetilde{G}_v$ and let $\alpha \in \widehat{\varphi}_1(\ker \widehat{\psi}_2^{v})$. We can write $\alpha = \widehat{\varphi}_1(\bigoplus\limits_{j=1}^{r} \widetilde{h}_j)=\prod\limits_{j=1}^{r} \widetilde{h}_j$ for some $\widetilde{h}_j \in \widetilde{H} \cap \widetilde{y}_j \widetilde{S}_v \widetilde{y}_j^{-1}$ such that $\beta:=\widehat{\psi}_2(\bigoplus\limits_{j=1}^{r} \widetilde{h}_j)=\prod\limits_{j=1}^{r} \widetilde{y}_j^{-1} \widetilde{h}_j \widetilde{y}_j$ is in $[\widetilde{S}_v,\widetilde{S}_v]$. Note that as $G$ is abelian, we have $[\widetilde{G},\widetilde{G}] \subset \widetilde{M}$ and therefore $[\widetilde{S}_v,\widetilde{S}_v] \subset \widetilde{M} \subset \widetilde{H}_i$ for every $1 \leq i \leq n$. In particular, we have $\beta \in \widetilde{H}_1 \cap \widetilde{S}_v$ and from this one readily checks that the $n$-tuple $(\beta,1,\dots,1) $ is in $\widetilde{\varphi}_1(\ker \widetilde{\psi}_2^{v})$. Since $\widetilde{H} \trianglelefteq \widetilde{G}$, we have $\Phi^{\widetilde{G}}(\widetilde{H})=[\widetilde{H},\widetilde{G}]$ and thus $f(\beta,1,\dots,1)=\beta=\prod\limits_{j} \widetilde{y}_j^{-1} \widetilde{h}_j \widetilde{y}_j \equiv \prod\limits_{j}  \widetilde{h}_j = \alpha \pmod{\Phi^{\widetilde{G}}(\widetilde{H})}$. As $\Phi^{\widetilde{G}}(\widetilde{H})/[\widetilde{H},\widetilde{H}]= \widehat{\varphi}_1(\ker \widehat{\psi}_2^{nr})$, we obtain $\alpha = f(\beta,1,\dots,1)$ inside $\widehat{\varphi}_1(\ker \widehat{\psi}_2) / \widehat{\varphi}_1(\ker \widehat{\psi}_2^{nr})$.\end{proof}

\begin{remark}
Note that the conditions of Theorem \ref{thm:demarch_wei_thm} are always satisfied if $n=2$, so that Theorem \ref{thm:pollio} generalizes the main theorem of \cite{pollio}.
\end{remark}

\subsection{Weak approximation for cyclic extensions of prime degree}\label{sec:eva}

\hspace{1pt}

In this subsection we extend the result in \cite[Theorem 8.3]{eva} to include the weak approximation property for the multinorm one torus of $n$ cyclic extensions of prime degree $p$. 

\begin{theorem}\label{thm:eva}
Let $L_1,\dots,L_n$ be non-isomorphic cyclic extensions of $K$ with prime degree $p$. Then, we have
\[\operatorname{H}^1(K,\Pic \overline{X})=\begin{cases}
(\Z/p)^{n-2}  \textrm{, if $[L_1\dots L_n:K]=p^2$;}\\
0\textrm{, otherwise.}
\end{cases}\] 
\end{theorem}

\begin{proof}
\noindent The case $n=1$ was proved in \cite[Proposition 9.1]{coll2} and for $n=2$ the result follows from Theorem \ref{thm:demarch_wei_thm}, so assume $n \geq 3$. 

Suppose first that $[L_1 \dots L_n : K] > p^2$. Reordering the fields $L_3,\dots,L_n$ if necessary, we can (and do) assume that each one of the fields $L_1,\dots,L_{s-1}$ is contained in $L_1 L_2$ (for some $3 \leq s\leq n$), while none of $L_s,\dots,L_n$ is contained in $L_1 L_2$. 
We prove two auxiliary claims:
\medskip

\textbf{Claim 1:} $\widetilde{H_i} \subset (\widetilde{H_1} \cap \widetilde{H_i} ).\widetilde{H_s}$ for any $i=1,\dots,s-1$.

\textbf{Proof:} Observe that $L_1L_i \cap L_s = K$ as otherwise we would have $L_s \subset L_1 L_i \subset L_1 L_2$, contradicting the assumption on $s$. Therefore $L_i \supset K= L_1 L_i \cap L_s$ and passing to subgroups this implies that ${H_i} \subset ({H_1} \cap {H_i}).{H_s}$, from which the claim follows.
\medskip

\textbf{Claim 2:} $\widetilde{H_i} \subset (\widetilde{H_1} \cap \widetilde{H_i}).\widetilde{H_2}$ for any $i=s,\dots,n$.

\textbf{Proof:} Observe that $L_2 \not\subset L_1 L_i$ as otherwise we would have $L_i \subset L_1 L_i = L_1 L_2$, contradicting the assumption on $L_i$. Therefore $L_i \supset K= L_1 L_i \cap L_2$ and passing to subgroups this implies that ${H_i} \subset ({H_1} \cap {H_i}).{H_2}$, from which the claim follows.

\medskip

Let us now prove that $\operatorname{H}^1(K,\Pic \overline{X})=0$. Since $\bigcap\limits_{i} L_i = K$, by Lemma \ref{lem:surject_int_Galois}\ref{lem:surject_int_Galois1} it suffices to show that $$\widetilde{\varphi}_1(\ker \widetilde{\psi}_2^{nr}) \supseteq \{ (h_1 \widetilde{H_1}' , \dots , h_n \widetilde{H_n}') | h_1\dots h_n \in \Phi^{\widetilde{G}}(\widetilde{H}) \}.$$

\noindent Let $\alpha=(h_1 \widetilde{H_1}' , \dots , h_n \widetilde{H_n}')$ be such that $h_1 \dots h_n \in \Phi^{\widetilde{G}}(\widetilde{H})$. By Claim 1 above, for $i=3,\dots,s-1$ we can write $h_i = h_{1,i}h_{s,i}$, where $h_{1,i} \in \widetilde{H_1} \cap \widetilde{H_i}$ and $h_{s,i} \in \widetilde{H_s} \cap \widetilde{H_i}$. Using this decomposition, we can apply Lemma \ref{lem:simpl_inters} as done in the proof of Theorem \ref{thm:demarch_wei_thm} in order to simplify $\alpha$ modulo $\widetilde{\varphi}_1(\ker \widetilde{\psi}_2^{nr})$ and assume it has the form $(h_1',h_2,1,\dots,1,h_s',h_{s+1}\dots,h_n)$ for some  $h_1' \in \widetilde{H_1},h_s' \in \widetilde{H_s}$. Using Claim 2 and Lemma \ref{lem:simpl_inters} in the same way, we further reduce $\alpha$ modulo $\widetilde{\varphi}_1(\ker \widetilde{\psi}_2^{nr})$ to a vector of the form $(h_1'',h_2',1,\dots,1)$ for some $h_1'' \in \widetilde{H_1},h_2' \in \widetilde{H_2}$ such that $h_1'' h_2' \in \Phi^{\widetilde{G}}(\widetilde{H})$. Finally, since $L_1 \cap L_2 = K$, we have $\widetilde{H}= \widetilde{H_1} \widetilde{H_2}$ and thus $\Phi^{\widetilde{G}}(\widetilde{H}) \subset \Phi^{\widetilde{G}}(\widetilde{H_1})\Phi^{\widetilde{G}}(\widetilde{H_2})$. The result follows by an argument similar to the one given at the end of the proof of Theorem \ref{thm:demarch_wei_thm}.

Now assume that $[L_1\dots L_n:K]=p^2$ (note that this is only possible if $ n \leq p+1$ as a bicyclic field has $p+1$ subfields of degree $p$) and therefore $G=C_p \times C_p$ is abelian. By Proposition \ref{prop:sq_free_mid_gp} it suffices to prove that $\ker {\psi}_1 / {\varphi}_1(\ker {\psi}_2^{nr}) \cong (\Z/p)^{n-2}$. We first show that ${\varphi}_1(\ker {\psi}_2^{nr})=1$. Let $\alpha \in \ker {\psi}_2^{v}$ for some unramified place $v$ of $N/K$. Write $G_v=\langle g \rangle$ and $\alpha = \bigoplus\limits_{i=1}^{n} \bigoplus\limits_{k=1}^{r_{v,i}} {h}_{i,k} $ for some $g \in G$ and ${h}_{i,k}  \in {H_i} \cap {x}_{i,k} \langle g \rangle {x}_{i,k}^{-1}={H_i} \cap \langle g \rangle$. If $g \not\in H_i $ for all $i=1,\dots,n$, then $\alpha$ is the trivial vector and $\varphi_1(\alpha)=(1,\dots,1)$. Otherwise, if $g \in H_{i_0}\cong C_p$ for some index $i_0$, then $g \not\in H_i$ for all $i \neq i_0$ and thus $h_{i,k}=1$ for $i \neq i_0$. In this way, it follows that $1=\psi_2(\alpha)=\prod\limits_{i=1}^{n} \prod\limits_{k=1}^{r_{v,i}} {x}_{i,k}^{-1} {h}_{i,k}{x}_{i,k}= \prod\limits_{k=1}^{r_{v,i_0}}  {h}_{i_0,k}$. Therefore, if $\varphi_1(\alpha)=(h_1,\dots,h_n)$, we have $h_i = 1$ if $i \neq i_0$ and $h_{i_0}=\prod\limits_{k=1}^{r_{v,i_0}}  {h}_{i_0,k} = 1$. In conclusion, $\varphi_1(\alpha)=(1,\dots,1)$.

On the other hand, we have $\ker \psi_1 = \{(h_1,\dots,h_n) | h_i \in H_i, \prod\limits_{i=1}^{n} h_i = 1\}$. This group is the kernel of the surjective group homomorphism 
\begin{align*}
  f \colon H_1 \times \dots \times H_n &\longrightarrow G \\
  (h_1 ,\dots,h_n ) &\longmapsto h_1 \dots h_n 
\end{align*}

\noindent and thus $\ker \psi_1 = \ker f \cong (\Z/p)^{n-2}$, as desired.\end{proof}

\begin{corollary}\label{cor:eva}
Let $L=(L_1,\dots,L_n)$ be an $n$-tuple of non-isomorphic cyclic extensions of $K$ with prime degree $p$. 
\begin{enumerate}
    \item\label{cor410_case1} If $[L_1\dots L_n:K]=p^2$, then weak approximation for the multinorm one torus $T$ holds if and only if the multinorm principle for $L$ fails.
    \item Otherwise, both the multinorm principle for $L$ and weak approximation for $T$ hold.
\end{enumerate}
 
\end{corollary}

\begin{proof}
Follows from Voskresenski\u{\i}'s exact sequence \eqref{eq:Vosk}, Theorem \ref{thm:eva} and \cite[Theorem 8.3]{eva}.
\end{proof}

\begin{remark}\label{rem:eva}
In \cite[Proposition 8.5]{eva} it is shown that, in the case \eqref{cor410_case1} above, the multinorm principle for $L$ fails if and only if all decomposition groups of the bicyclic extension $L_1 \dots L_n$ are cyclic. We thus have a simple criterion to test the validity of weak approximation for the associated multinorm one torus.
\end{remark}

\end{document}